\documentclass[a4paper,10pt]{amsart}
\usepackage[left=2.7cm,right=2.7cm,top=3.5cm,bottom=3cm]{geometry}

\usepackage{amssymb,latexsym,amsmath,amsthm,amscd}
\usepackage{graphicx}
\usepackage{dsfont}
\usepackage[all]{xy}
\usepackage{mathrsfs}
\usepackage{color}
\definecolor{Blue}{rgb}{0.3,0.3,0.9}

\usepackage{amsmath}
\usepackage{amscd,amsthm,amssymb,amsfonts}
\usepackage{mathrsfs}
\usepackage{dsfont}
\usepackage{stmaryrd}
\usepackage{euscript}
\usepackage{expdlist}
\usepackage{enumerate}

\newcommand{\sk}{\vspace{0.1in}}
\newtheorem{thm}{Theorem}[section]
\newtheorem{def-thm}[thm]{Definition-Theorem}
\newtheorem{cor}[thm]{Corollary}
\newtheorem{lem}[thm]{Lemma}
\newtheorem{prop}[thm]{Proposition}
\newtheorem{ass}[thm]{Assumptions}

\newtheorem*{mainthmA}{Theorem}

\theoremstyle{definition}
\newtheorem{defn}[thm]{Definition}
\theoremstyle{remark}
\newtheorem{rem}[thm]{Remark}
\newtheorem{intro-rem}{Remark}
\numberwithin{equation}{section}

\newcommand{\A}{\mathcal{A}}
\newcommand{\cM}{\mathcal{M}}
\newcommand{\cV}{\mathcal{V}}
\newcommand{\cW}{\mathcal{W}}
\newcommand{\cO}{\mathcal{O}}

\newcommand{\cU}{\mathcal{U}}
\newcommand{\cX}{\mathcal{X}}
\newcommand{\cZ}{\mathcal{Z}}
\newcommand{\cE}{\mathcal{E}}
\newcommand{\Ext}{{\rm Ext}}

\newcommand{\pp}{\mathfrak{p}}
\newcommand{\qq}{\mathfrak{q}}
\newcommand{\fa}{\mathfrak{a}}

\newcommand{\ccL}{\mathcal{L}}
\newcommand{\cI}{\mathcal{I}}
\newcommand{\bQ}{\mathbf{Q}}
\newcommand{\bZ}{\mathbf{Z}}
\newcommand{\bC}{\mathbf{C}}

\newcommand{\bT}{\mathbb{T}}

\newcommand{\hcris}{H_{\textrm{log-cris}}}

\newcommand{\cR}{\mathbb{I}}
\newcommand{\fil}{{\mathscr{F}}}

\newcommand{\sF}{\mathbf{f}}
\newcommand{\F}{\mathbf{f}}
\newcommand{\pwseries}[1]{[[#1]]}
\newcommand{\arrow}{\longrightarrow}

\begin{document}

\title[On the exceptional specializations of big Heegner points]
{On the exceptional specializations of big Heegner points\\
}
\author[F.~Castella]{Francesc Castella}
\address{Department of Mathematics, UCLA, Math Sciences Building, Los Angeles, CA 90095, USA}
\email{castella@math.ucla.edu}

\thanks{Research supported in part by Grant MTM2012-34611
and by Prof.~Hida's NSF Grant DMS-0753991.}

\date{\today}



\begin{abstract}
We extend the $p$-adic Gross--Zagier formula of Bertolini, Darmon, and Prasanna \cite{bdp1}
to the semistable non-crystalline setting, and
combine it with our previous work \cite{cas-2var}
to obtain a derivative formula for the specializations of Howard's big Heegner points \cite{howard-invmath}
at exceptional primes in the Hida family.
\end{abstract}


\maketitle
\tableofcontents

\section*{Introduction}

Fix a prime $p\geq 5$, an integer $N>0$ prime to $p$, and
let $f\in S_2(\Gamma_0(Np))$ be a newform.
Throughout this paper, we shall assume that $f$ is \emph{split multiplicative} at $p$,
meaning that
\[
f(q)=q+\sum_{n=2}^\infty a_n(f)q^n\quad\quad\textrm{with $a_p(f)=1$.}
\]
Fix embeddings $\bC\overset{\imath_\infty}\hookleftarrow\overline{\bQ}\overset{\imath_p}\hookrightarrow\bC_p$,
let $L$ be a finite extension of $\bQ_p$ containing $\imath_p\imath_\infty^{-1}(a_n(f))$ for all $n$,
and let $\cO_L$ be the ring of integers of $L$. Since the $U_p$-eigenvalue of $f$ is $a_p(f)=1$ by hypothesis,
the form $f$ is ordinary at $p$, and hence there is a Hida family
\[
\mathbf{f}=\sum_{n=1}^\infty\mathbf{a}_nq^n\in\cR\pwseries{q}
\]
passing through $f$. Here $\cR$ is a finite flat extension of
the power series ring $\cO_L\pwseries{T}$, which for simplicity in this introduction
it will be assumed to be $\cO_L\pwseries{T}$ itself. Embed $\bZ$ in the space
$\mathcal{X}_{\cO_L}(\cR)$ of continuous $\cO_L$-algebra homomorphisms
$\nu:\cR\longrightarrow\overline{\bQ}_p$ by identifying $k\in\bZ$
with the homomorphism $\nu_k:\cR\longrightarrow\overline{\bQ}_p$
defined by $1+T\mapsto(1+p)^{k-2}$. The Hida family $\mathbf{f}$ is then uniquely caracterized by the property that for every
$k\in\bZ_{\geq 2}$ its \emph{weight $k$ specialization}
\[
\mathbf{f}_k:=\sum_{n=1}^{\infty}\nu_k(\mathbf{a}_n)q^n
\]
gives the $q$-expansion of a $p$-ordinary
$p$-stabilized newform $\mathbf{f}_k\in S_k(\Gamma_0(Np))$ with $\mathbf{f}_{2}=f$.
\sk

Let $K$ be an imaginary quadratic field equipped with an integral ideal $\mathfrak{N}\subset\cO_K$
with $\cO_K/\mathfrak{N}\simeq\bZ/N\bZ$, assume that $p$ splits in $K$, and
write $p\cO_K=\mathfrak{p}\overline{\mathfrak{p}}$ with $\mathfrak{p}$ the prime above $p$ induced by $\imath_p$.
If $A$ is an elliptic curve with CM by $\cO_K$, then the pair $(A,A[\mathfrak{Np}])$ defines a
\emph{Heegner point} $P_A$ on $X_0(Np)$ defined over the Hilbert class field $H$ of $K$.
Taking the image of the degree zero divisor $(P_A)-(\infty)$ under the composite map
\begin{equation}\label{def:heeg}
J_0(Np)\xrightarrow{\rm Kum}H^1(H,{\rm Ta}_p(J_0(Np)))\longrightarrow H^1(H,V_f)\xrightarrow{{\rm Cor}_{H/K}} H^1(K,V_f)
\end{equation}
yields a class $\kappa_f\in{\rm Sel}(K,V_f)$ in the Selmer group for the
$p$-adic Galois representation
\[
\rho_f:G_{\bQ}:={\rm Gal}(\overline{\bQ}/\bQ)\arrow{\rm Aut}_L(V_f)\simeq\mathbf{GL}_2(L)
\]
associated to $f$.
On the other hand, by working over a $p$-tower of modular curves,
Howard~\cite{howard-invmath} constructed a so-called \emph{big Heegner point}
$\mathfrak{Z}_0\in{\rm Sel}_{\rm Gr}(K,\mathbf{T}^\dagger)$ in the Selmer group
for a self-dual twist of the big Galois representation
\[
\rho_{\sF}:G_\bQ\arrow{\rm Aut}_{\cR}(\mathbf{T})\simeq\mathbf{GL}_2(\cR)
\]
associated to $\sF$. The image of $\mathfrak{Z}_0$ under the
\emph{specialization map} $\nu_2:{\rm Sel}_{\rm Gr}(K,\mathbf{T}^\dagger)\longrightarrow{\rm Sel}(K,V_f)$
induced by $\nu_2:\cR\longrightarrow\overline{\bQ}_p$ yields a second class of ``Heegner type'' in ${\rm Sel}(K,V_f)$;
the question of comparing $\kappa_f$ with $\nu_2(\mathfrak{Z}_0)$ thus naturally arises.
\sk

For $k>2$, the question of relating the specializations $\nu_k(\mathfrak{Z}_0)$ to higher dimensional Heegner cycles
was considered in \cite{cas-inv}. In that case, one could show (see [\emph{loc.cit.}, (5.31)]) that
\begin{equation}\label{eq:mathann}
{\rm loc}_\pp(\nu_k(\mathfrak{Z}_0))=u^{-1}\biggl(1-\frac{p^{k/2-1}}{\nu_k(\mathbf{a}_p)}\biggr)^2\cdot
{\rm loc}_\pp(\kappa_{\mathbf{f}_k}),
\end{equation}
where $u:=\vert\cO_K^\times\vert/2$,
${\rm loc}_{\mathfrak{p}}:H^1(K,V_{\mathbf{f}_k})\longrightarrow H^1(K_{\mathfrak{p}},V_{\mathbf{f}_k})$
is the localization map, and $\kappa_{\mathbf{f}_k}$ is a class given by the
$p$-adic \'etale Abel--Jacobi images of certain Heegner cycles on a
Kuga--Sato variety of dimension $k-1$. However, for the above newform $f$,
the main result of \cite{cas-inv} does not immediately yield a similar
relation between $\nu_2(\mathfrak{Z}_0)$ and $\kappa_{\mathbf{f}_2}=\kappa_f$,
since in \emph{loc.cit.} a crucial use is made of the fact
that the $p$-adic Galois representations associated with the eigenforms under consideration
are (potentially) crystalline at $p$, whereas $V_f$ is well-known to be semistable but non-crystalline at $p$.
Moreover, it is easy to see that the expected relation between these two classes may not be
given by the naive extension of $(\ref{eq:mathann})$ with $k=2$: indeed, granted the injectivity of ${\rm loc}_\pp$,
by the Gross--Zagier formula the class ${\rm loc}_\pp(\kappa_f)$ is nonzero as long as $L'(f/K,1)\neq 0$,
whilst $(\ref{eq:mathann})$ for $k=2$ would imply the vanishing of ${\rm loc}_\pp(\nu_2(\mathfrak{Z}_0))$
is all cases, since
\begin{equation}\label{vanishing}
\biggl(1-\frac{p^{k/2-1}}{\nu_k(\mathbf{a}_p)}\biggr)\bigr\vert_{k=2}
=\left(1-\frac{1}{a_p(f)}\right)=0.
\end{equation}

As shown in \cite{howard-invmath}, the class $\mathfrak{Z}_0$ fits in compatible system of similar
classes $\mathfrak{Z}_\infty=\{\mathfrak{Z}_n\}_{n\geq 0}$ over
the anticyclotomic $\bZ_p$-extension of $K$; thus $\mathfrak{Z}_0$ might be seen
as the value of $\mathfrak{Z}_\infty$ at the trivial character.
As suggested by the above discussion, in this paper we will show that the class
${\rm loc}_\pp(\nu_2(\mathfrak{Z}_0))$ vanishes, and prove an ``exceptional zero formula''
relating its derivative at the trivial character (in a precise sense to be defined) to the geometric class $\kappa_f$.
To state the precise result, 
let $h$ be the class number of $K$, write $\pp^{h}=\pi_\pp\cO_K$, 
and define
\begin{equation}\label{def:Linv}
\mathscr{L}_\pp(f,K):=\mathscr{L}_p(f)-\frac{\log_p(\varpi_\pp)}{{\rm ord}_p(\varpi_\pp)},
\end{equation}
where $\mathscr{L}_p(f)$ is the $\mathscr{L}$-invariant of $f$ (see \cite[\S{II.14}]{mtt} for example),
$\varpi_\pp:=\pi_\pp/\overline{\pi}_{\pp}\in K_\pp\simeq\bQ_p$,
and $\log_p:\bQ_p^\times\longrightarrow\bQ_p$ is Iwasawa's branch of the $p$-adic logarithm.

\begin{mainthmA}\label{intro:main}
Let $f\in S_2(\Gamma_0(Np))$ be a newform split multiplicative at $p$,
and define $\mathcal{Z}_{\pp,f,\infty}=\{\mathcal{Z}_{\pp,f,n}\}_{n\geq 0}$
by $\mathcal{Z}_{\pp,f,n}:={\rm loc}_\pp(\nu_2(\mathfrak{Z}_n))$.
Then $\mathcal{Z}_{\pp,f,0}=0$ and
\begin{equation}\label{intro:exczero}
\mathcal{Z}_{\pp,f,0}'
=\mathscr{L}_\pp(f,K)\cdot{\rm loc}_\pp(\kappa_f).\nonumber
\end{equation}
\end{mainthmA}

In Lemma~\ref{lem:divide} below, we define the ``derivative''
$\mathcal{Z}_{\infty}'$ for any compatible system of classes
$\mathcal{Z}_\infty=\{\mathcal{Z}_n\}_{n\geq 0}$ with $\mathcal{Z}_0=0$.
Thus the above result, which corresponds to Theorem~\ref{main} in the body of the paper,
might be seen as an exceptional zero formula
relating the derivative of ${\rm loc}_\pp(\nu_2(\mathfrak{Z}_\infty))$ at the trivial character
to classical Heegner points.

\begin{intro-rem}
As suggested in \cite[\S{8}]{LLZ}, one might view $p$-adic $L$-functions
(as described in \cite{PR:Lp} and \cite[Ch.~8]{Rubin-ES}) as
``rank 0" Euler--Iwasawa systems. In this view, it is natural to expect higher rank
Euler--Iwasawa systems to exhibit exceptional zero phenomena similar to their rank $0$ counterparts.
We would like to see the main result of this paper
as an instance of this phenomenon in ``rank $1$''.
\end{intro-rem}

\begin{intro-rem}
It would be interesting to study the formulation of our main result
in the framework afforded by Nekov{\'a}{\v{r}}'s theory of Selmer complexes \cite{nekovar310}, similarly as the
exceptional zero conjecture of Mazur--Tate--Teitelbaum \cite{mtt}  
has recently been proved by Venerucci \cite{venerucci-exp} in the rank $1$ case.
\end{intro-rem}

\begin{intro-rem}
The second term in the definition $(\ref{def:Linv})$ is
precisely the $\mathscr{L}$-invariant $\mathscr{L}_\pp(\chi_K)$ appearing in
the exceptional zero formula of Ferrero--Greenberg~\cite{FG} and Gross--Koblitz~\cite{GK}
for the Kubota--Leopoldt $p$-adic $L$-function associated to the quadratic Dirichlet character
$\chi_{K}$ corresponding to $K$. It would be interesting
to find a conceptual explanation for the rather surprising appearance of $\mathscr{L}_\pp(\chi_K)$
in our derivative formula; we expect this to be related to a comparison of $p$-adic periods (cf.~\cite{cas-BF}).
\end{intro-rem}


The proof of the above theorem is obtained by
computing in two different ways the value of a certain
anticyclotomic $p$-adic $L$-function $L_{\pp}(f)$
at the norm character ${\rm\mathbf{N}}_K$. The $p$-adic $L$-function $L_\pp(f)$
is defined by the interpolation of the central critical values for the Rankin--Selberg convolution of $f$
with the theta series attached to Hecke characters of $K$ of infinity type $(2+j,-j)$ with $j\geq 0$.
The character ${\rm\mathbf{N}}_K$ thus lies \emph{outside} the range of interpolation of $L_\pp(f)$,
and via a suitable extension of the methods of Bertolini--Darmon--Prasanna~\cite{bdp1}
to our setting, in Theorem~\ref{thmbdp1A} we show that
\begin{equation}\label{intro:pGZ}
L_\pp(f)({\rm\mathbf{N}}_K)
=(1-a_p(f)p^{-1})\cdot\langle{\rm log}_{V_f}({\rm loc}_\pp(\kappa_f)),\omega_{f}\rangle_{}.
\end{equation}

On the other hand, in \cite{cas-2var} we constructed a two-variable
$p$-adic $L$-function $L_{\pp,\xi}(\mathbf{f})$ of the variables $(\nu,\phi)$
interpolating (a shift of) the $p$-adic $L$-functions $L_\pp(\mathbf{f}_k)$ for all $k\geq 2$,
and established the equality
\begin{equation}
\label{intro:Log}
L_{\pp,\xi}(\sF)=\mathcal{L}_{\fil^+\bT}^\omega({\rm loc}_\pp(\mathfrak{Z}^{\xi^{-1}}_\infty)),
\end{equation}
where $\mathcal{L}_{\fil^+\bT}^\omega$ is a two-variable Coleman power series map whose
restriction to a certain ``line'' interpolates
\[
\biggl(1-\frac{p^{k/2-1}}{\nu_k(\mathbf{a}_p)}\biggr)^{-1}
\biggl(1-\frac{\nu_k(\mathbf{a}_p)}{p^{k/2}}\biggr)\cdot
{\rm log}_{V_{\mathbf{f}_k}}
\]
for all $k>2$. A second evaluation of $L_\pp(f)({\rm\mathbf{N}}_K)$ should thus follow by
specializing $(\ref{intro:Log})$ at $(\nu_2,\mathds{1})$.
However, because of the vanishing $(\ref{vanishing})$,  we may not directly specialize
$\mathcal{L}_{\fil^+\bT}^\omega$ at $(\nu_2,\mathds{1})$, and we are led to utilize a different argument
reminiscent of Greenberg--Stevens'~\cite{GS}. In fact, from the form of the $p$-adic multipliers
appearing in the interpolation property defining $\mathcal{L}_{\fil^+\bT}^\omega$,
we deduce a factorization
\begin{equation}\label{intro:imp}
E_\pp(\sF)\cdot L_{\pp,\xi}(\sF)=\widetilde{\mathcal{L}}_{\fil^+\bT}^\omega({\rm loc}_\pp(\mathfrak{Z}^{\xi^{-1}}_0))\nonumber
\end{equation}
upon restricting $(\ref{intro:Log})$ to an appropriate ``line'' (different from the above)
passing through $(\nu_2,\mathds{1})$, where $\widetilde{\mathcal{L}}_{\fil^+\bT}^\omega$ is a modification of
$\mathcal{L}_{\fil^+\bT}^\omega$ and $E_\pp(\sF)$ is a $p$-adic analytic
function vanishing at that point.
The vanishing of $\mathcal{Z}_{\pp,f,0}$ thus follows, and exploiting
the ``functional equation'' satisfied by $\mathfrak{Z}_\infty$,
we arrive at the equality
\begin{equation}\label{intro:lim}
\mathscr{L}_\pp(f,K)\cdot L_\pp(f)({\rm\mathbf{N}}_K)
=(1-a_p(f)p^{-1})\cdot\langle{\rm log}_{V_f}(\mathcal{Z}_{\pp,f,0}'),\omega_{f}\rangle_{}
\end{equation}
using a well-known formula for the $\mathscr{L}$-invariant as a
logarithmic derivative of $\nu_k(\mathbf{a}_p)$ at $k=2$. 
The proof of our exceptional zero formula then follows by combining $(\ref{intro:pGZ})$ and $(\ref{intro:lim})$.
\sk

\noindent\emph{Acknowledgements.}
We would like to thank Daniel Disegni for 
conversations related to this work, and the referee for pointing out a number of
inaccuracies in an earlier version of this paper,
as well as for helpful suggestions
which led to significant improvements in the article.

\section{Preliminaries}

For a more complete and detailed discussion of the topics that we touch upon in this section,
we refer the reader to \cite{pSI} and \cite{bdp1}. 

\subsection{Modular curves}\label{subsubsec:X}

Keep $N$ and $p\nmid N$ as in the Introduction, and let
\[
\Gamma:=\Gamma_1(N)\cap\Gamma_0(p)\subset\mathbf{SL}_2(\mathbf{Z}).
\]
An \emph{elliptic curve with $\Gamma$-level structure} over a $\bZ[1/N]$-scheme $S$ is a triple
$(E,t,\alpha)$ consisting of
\begin{itemize}
\item{} an elliptic curve $E$ over $S$;
\item{} a section $t:S\longrightarrow E$ of
the structure morphism of $E/S$ of exact order $N$; and
\item{} a $p$-isogeny $\alpha:E\longrightarrow E'$.
\end{itemize}

The functor on $\bZ[1/N]$-schemes assigning to $S$ the set of
isomorphism classes of elliptic curves with $\Gamma$-level
structure over $S$ is representable, and we let $Y/\bZ[1/N]$ be the
corresponding fine moduli scheme. The same
moduli problem for \emph{generalized} elliptic curves
with $\Gamma$-level structure defines
a smooth geometrically connected curve $X/\bZ[1/N]$ containing
$Y$ as a open subscheme, and we refer to $Z_X:=X\smallsetminus Y$
as the \emph{cuspidal subscheme} of $X$. Removing the data of
$\alpha$ from the above moduli problem, we obtain the modular curve $X_1(N)$
of level $\Gamma_1(N)$.
\sk

For our later use (see esp.~Theorem~\ref{coleman-primitives}),
recall that if $a$ is any integer coprime to $N$, the rule
\[
\langle a\rangle(E,t,\alpha)=(E,a\cdot t,\alpha)
\]
defines an action of $(\bZ/N\bZ)^\times$ on $X$ defined over $\bZ[1/N]$,
and we let $X_0(Np)=X/(\bZ/N\bZ)^\times$ be the quotient of $X$ by this action.
\sk


The special fiber $X_{\mathbf{F}_p}:=X\times_{\bZ[1/N]}\mathbf{F}_p$
is non-smooth. In fact, it consists of two irreducible components,
denoted $C_0$ and $C_\infty$, meeting transversally at the singular points $SS$.
Let ${\rm Frob}$ be the absolute Frobenius of an elliptic curve over $\mathbf{F}_p$,
and ${\rm Ver}={\rm Frob}^\vee$ be the Verschiebung. The maps
\[
\gamma_{V}:X_1(N)_{\mathbf{F}_p}:=X_1(N)\times_{\bZ[1/N]}\mathbf{F}_p\arrow X_{\mathbf{F}_p}
\quad\quad
\gamma_{F}:X_1(N)_{\mathbf{F}_p}\arrow X_{\mathbf{F}_p}
\]
defined by sending a pair $(E,t)_{/\mathbf{F}_p}$ to $(E,t,{\rm ker}({\rm Ver}))$
and $(E,t,{\rm ker}({\rm Frob}))$ respectively, are closed immersions sending
$X_1(N)_{\mathbf{F}_p}$ isomorphically onto $C_0$ and $C_\infty$,
and mapping the supersingular points in $X_1(N)_{\mathbf{F}_p}$ bijectively onto $SS$.
The non-singular geometric points of $C_0$ (resp. $C_\infty$)
thus correspond to the moduli of triples $(E,t,\alpha)$ in characteristic $p$ with
${\rm ker}(\alpha)$ \'etale (resp. connected).
\sk

Corresponding to the preceding description of $X_{\mathbf{F}_p}$
there is a covering of $X$ as rigid analytic space over $\bQ_p$.
Consider the reduction map
\begin{equation}\label{eq:red}
{\rm red}_p:X(\bC_p)\arrow X_{\mathbf{F}_p}(\overline{\mathbf{F}}_p),
\end{equation}
let $\cW_0$ and $\cW_\infty$ be the inverse image of $C_0$ and $C_\infty$, respectively,
and let $\cZ_0\subset\cW_0$ and $\cZ_\infty\subset\cW_\infty$ be the inverse image of their non-singular points.
In the terminology of \cite{rlc}, $\cW_0$ (resp. $\cW_\infty$) is a \emph{basic wide open}
with underlying affinoid $\cZ_0$ (resp. $\cZ_\infty$).
If $x\in SS$, then $\A_x:={\rm red}_p^{-1}(x)$ is conformal to an open
annulus in $\bC_p$, and by definition we have
\[
X(\bC_p)=\cW_0\cup\cW_\infty=\cZ_0\cup\cZ_\infty\cup\cW,
\]
where $\cW=\cW_0\cap\cW_\infty=\bigcup_{x\in SS}\A_x$ is the union of the supersingular annuli.

\subsection{Modular forms and cohomology}\label{subsubsec:mf&dR}

In this section, we regard the modular curve $X$ as a scheme over a fixed base field $F$.
Let $\cE\xrightarrow{\;\pi\;}X$ be the universal
generalized elliptic curve with $\Gamma$-level structure,
set $\tilde{Z}_X=\pi^{-1}(Z_X)$, and consider the invertible sheaf on $X$
given by
\[
\underline{\omega}:=\pi_*\Omega_{\cE/X}^1(\log\widetilde{Z}_X).
\]

The space of algebraic \emph{modular
forms} (resp. \emph{cusp forms}) of weight $k$
and level $\Gamma$ defined over $F$ is
\[
M_k(X;F):=H^0(X,\underline{\omega}_F^{\otimes k})\quad\quad
(\textrm{resp.}\; S_k(X;F):=H^0(X,\underline{\omega}_F^{\otimes
k}\otimes\cI)),
\]
where $\underline{\omega}_F$ is the pullback of
$\underline{\omega}$ to $X\times_{\bQ}F$,
and $\cI$ is the ideal sheaf of $Z_X\subset X$.
If there is no risk of confusion, $F$ will be often suppressed from the notation.
Alternatively, on the open modular curve $Y$ a form
$f\in S_k(X;F)\subset M_k(X;F)$ is a rule
on quadruples $(E,t,\alpha,\omega)_{/A}$,
consisting of an $A$-valued point $(E,t,\alpha)\in Y(A)$
and a differential $\omega\in\Omega^1_{E/A}$ over arbitrary $F$-algebras $A$,
assigning to any such quadruple a value $f(E,t,\alpha,\omega)\in A$ subject
to the \emph{weight $k$ condition}
\[
f(E,t,\alpha,\lambda\omega)=\lambda^{-k}\cdot f(E,t,\alpha,\omega)\;\;\;
\textrm{for all $\lambda\in A^\times$,}
\]
depending only on the isomorphism class of the quadruple,
and compatible with base change of $F$-algebras.
The two descriptions are related by
\[
f(E,t,\alpha)=f(E,t,\alpha,\omega)\omega^k,
\]
for any chosen generator $\omega\in\Omega_{E/A}^1$.
\sk

There is a third way of thinking about modular forms that
will be useful in the following. Consider the
\emph{relative de Rham cohomology} of $\cE/X$:
\[
\ccL:=\mathbb{R}^1\pi_*(0\arrow\cO_\cE
\arrow\Omega_{\cE/X}^1(\log\widetilde{Z}_X)\arrow 0),
\]
which fits in a short exact sequence
\begin{equation}\label{hodgedeRham}
0\arrow\underline{\omega}\arrow\ccL\arrow
\underline{\omega}^{-1}\arrow 0
\end{equation}
of sheaves on $X$ and is equipped with
a non-degenerate pairing
\begin{equation}\label{PDL}
\langle,\rangle:\ccL\times\ccL\arrow\mathcal{O}_X
\end{equation}
coming from the Hodge filtration and the Poincar\'e pairing
on the de Rham cohomology of the fibers. By the Kodaira--Spencer isomorphism
\begin{equation}\label{KS}
\sigma:\underline{\omega}^{\otimes 2}\cong\Omega_X^1(\log Z_X)\nonumber
\end{equation}
given by $\sigma(\omega\otimes\eta)=\langle\omega,\nabla\eta\rangle$, where
\[
\nabla:\ccL\arrow\ccL\otimes\Omega_X^1(\log Z_X)
\]
is the Gauss--Manin connection, a modular form $f$
of weight $r+2$ and level $\Gamma$ defines a section $\omega_f$ of the sheaf
$\underline{\omega}^{\otimes r}\otimes\Omega_X^1(\log Z_X)$ by the rule
\begin{equation}\label{f-wf}
\omega_f(E,t,\alpha):=f(E,t,\alpha,\omega)\omega^r\otimes\sigma(\omega^2).\nonumber
\end{equation}
If $f$ is a cusp form, then the above rule defines a section
$\omega_f$ of $\underline{\omega}^{\otimes r}\otimes\Omega_X^1$, thus
yielding an identification
\[
S_{r+2}(X)\simeq H^0(X,\underline{\omega}^{\otimes r}\otimes\Omega_X^1).
\]
For each $r\geq 0$, let $\ccL_r:={\rm Sym}^{r}\ccL$ (with $\ccL_0:=\cO_X$),
and define the \emph{de Rham cohomology} of $X$ (attached to $\ccL_r$) as
the hypercohomology group
\begin{equation}\label{def:HdR}
H_{\rm dR}^1(X,\ccL_r,\nabla):=\mathbb{H}^1(\ccL_r^\bullet:
\ccL_r\xrightarrow{\;\nabla\;}\ccL_r\otimes
\Omega_X^1(\log Z_X)).
\end{equation}
Twisting by the ideal sheaf $\cI$ gives rise to
the subcomplex $\ccL_r^\bullet\otimes\cI\longrightarrow\ccL_r^\bullet$,
and the weight $r+2$ \emph{parabolic cohomology} of $X$ is defined by
\begin{equation}\label{def:Hpar}
H^1_{\rm par}(X,\ccL_r,\nabla)
:={\rm image}(\mathbb{H}^1(\ccL_r^\bullet\otimes\cI)\arrow
H_{\rm dR}^1(X,\ccL_r,\nabla)).
\end{equation}

The exact sequence $(\ref{hodgedeRham})$ induces the short exact sequence
\begin{equation}\label{hodgefilpar}
0\arrow H^0(X,\underline{\omega}^{\otimes r}\otimes\Omega_X^1)\arrow H_{\rm
par}^1(X,\ccL_r,\nabla)\arrow
H^1(X,\underline{\omega}^{\otimes-r})\arrow 0,
\end{equation}
and hence the above assignment $f\mapsto\omega_f$
identifies $S_{r+2}(X)$ with a subspace of $H^1_{\rm par}(X,\ccL_r,\nabla)$.
In addition, the pairing $(\ref{PDL})$ induces a non-degenerate pairing
\begin{equation}\label{PDpar}
\langle,\rangle:H_{\rm par}^1(X,\ccL_r,\nabla)\times
H_{\rm par}^1(X,\ccL_r,\nabla)\arrow F
\end{equation}
with respect to which $(\ref{hodgefilpar})$ is self-dual.

\subsection{$p$-new forms}

Consider the two degeneracy maps
\begin{align*}
\pi_1, \pi_2&:X\arrow X_1(N)
\end{align*}
defined by sending, under the moduli interpretation,
a triple $(E,t,\alpha)$ to the pairs
$(E,t)$ and $(\alpha(E),\alpha(t))$, respectively.
These morphisms induce maps
\begin{align*}
\pi_1^*,\pi_2^*&:H_{\rm par}^1(X_1(N),\ccL_r,\nabla)\arrow H_{\rm
par}^1(X,\ccL_r,\nabla),
\end{align*}
where $H^1_{\rm par}(X_1(N),\ccL_r,\nabla)$ is defined as in $(\ref{def:Hpar})$
using the analogous objects over $X_1(N)$.

\begin{lem}
The map $\pi_1^*\oplus\pi_2^*$ 
is injective.
\end{lem}

\begin{proof}
This is \cite[Prop.~4.1]{pSI}.
\end{proof}

Define the \emph{$p$-old} subspace $H_{\rm
par}^1(X,\ccL_r,\nabla)^{p{\rm -old}}$ of
$H_{\rm par}^1(X,\ccL_r,\nabla)$ to be the image of $\pi_1^*\oplus\pi_2^*$,
and the \emph{$p$-new} subspace $H_{\rm par}^1(X,\ccL_r,\nabla)^{p-{\rm new}}$
to be the orthogonal complement of the $p$-old subspace under the Poincar\'e pairing $(\ref{PDpar})$.
The space of $p$-new cusp forms of weight $k$ and level $\Gamma$ is defined by
\[
S_{r+2}(X)^{p{\rm -new}}:=S_{r+2}(X)\cap H_{\rm
par}^1(X,\ccL_r,\nabla)^{p{\rm -new}},
\]
viewing $S_{r+2}(X)$ as subspace of $H_{\rm par}^1(X,\ccL_r,\nabla)$
in the form described above.

\subsection{$p$-adic modular forms}\label{sec:p-adic}

Recall that the \emph{Hasse invariant} is a modular form $H$ over $\mathbf{F}_p$ 
of level $1$ and weight $p-1$ with the property that an elliptic curve $E$ over
an $\mathbf{F}_p$-algebra $B$ is \emph{ordinary} if and only if
$H(E,\omega)$ is a unit in $B$ for some (or equivalently, any) generator $\omega\in\Omega_{E/B}^1$.
\sk

Let $R$ be a $p$-adic ring, i.e., a ring which is isomorphic to its pro-$p$ completion.
A \emph{$p$-adic modular form} of tame level $N$ and weight $k$ defined over $R$ is
a rule assigning to every triple $(E,t,\omega)_{/A}$, over an arbitrary $p$-adic $R$-algebra $A$,
consisting of:
\begin{itemize}
\item{} an elliptic curve $E/A$ such that the reduction $E\times_A A/pA$ is ordinary;
\item{} a section $t:{\rm Spec}(A)\longrightarrow E$ of
the structure morphism of $E/A$ of exact order $N$; and
\item{} a differential $\omega\in\Omega_{E/A}^1$,
\end{itemize}
an element $f(E,t,\omega)\in A$ depending only on the isomorphism class
of $(E,t,\omega)_{/A}$, homogeneous of degree $-k$ in the third
entry, and compatible with base change of $p$-adic $R$-algebras.
Let $\cM_k(N;R)$ be the $R$-module of $p$-adic modular forms of weight $k$
and level $N$ defined over $R$; as before, if there is no risk of confusion
$R$ will be often suppressed from the notation.
\sk

Similarly as for classical modular forms, it will be convenient to think of
$p$-adic modular forms of weight $k$ as sections of the sheaf $\underline{\omega}^{\otimes k}$
over a certain subset of the rigid analytic space $X(\bC_p)$.
Let $E_{p-1}$ be the normalized  Eisenstein series of weight $p-1$ (recall that $p\geq 5$),
and define the \emph{ordinary locus} of $X_1(N)$ by
\[
X_1(N)^{\rm ord}:=\{x\in X_1(N)(\bC_p)\;\colon\;\vert E_{p-1}(E_x,\omega_x)\vert_p\geq 1\},
\]
where $E_x/\bC_p$ is a generalized elliptic curve corresponding to $x$
under the moduli interpretation, $\omega_x\in\Omega_{E_x/\bC_p}^1$
is a regular differential on $E_x$, chosen so that it extends to a regular
differential over $\cO_{\bC_p}$ if $E_x$ has good reduction at $p$,
or corresponds to the canonical differential on the Tate curve if $x$
lies in the residue disc of a cusp, and $\vert\cdot\vert_p$ is the absolute value on $\bC_p$
normalized so that $\vert p\vert_p=p^{-1}$.
Since $E_{p-1}$ reduces to the Hasse invariant $H$ modulo $p$,
it follows that the points $x\in X_1(N)^{\rm ord}$
correspond to pairs $(E_x,t_x)$ with $E_x$ having ordinary reduction modulo $p$.
Thus the assignment $f\mapsto (x\mapsto f(E_x,t_x,\omega_x)\omega_x^k)$, for any chosen
generator $\omega_x\in\Omega_{E_x/\bC_p}^1$, defines an identification
\[
\cM_k(N)\simeq H^0(X_1(N)^{\rm ord},\underline{\omega}^{\otimes k}).
\]
%
%

Let $I:=\{v\in\bQ\;\colon\;0< v\leq\frac{p}{p+1}\}$,
and for any $v\in I$ define
\[
X_1(N)(v):=\{x\in X_1(N)(\bC_p)\;\colon\;\vert E_{p-1}(E_x,\omega_x)\vert_p>p^{-v}\}.
\]
The space of \emph{overconvergent $p$-adic modular forms}
of weight $k$ and tame level $N$ is given by
\[
\cM_k^\dagger(N)=\varinjlim_{v}H^0(X_1(N)(v),\underline{\omega}^{\otimes k}),
\]
where the transition maps
$H^0(X_1(N)(v),\underline{\omega}^{\otimes k})
\longrightarrow H^0(X_1(N)(v'),\underline{\omega}^{\otimes k})$,
for $v'<v$ in $I$, are given by restriction; since these maps are injective,
$\cM_k^\dagger(N)$ is naturally a subspace of $\cM_k(N)$.
\sk

By the theory of the \emph{canonical subgroup}
(see \cite[Thm.~3.1]{Katz350}), if $(E_x,t_x)$ corresponds to a point $x$
in $X_1(N)(\frac{p}{p+1})$, the elliptic curve $E_x$ admits a distinguished subgroup
${\rm can}(E_x)\subset E_x[p]$ of order $p$ reducing to the kernel of
Frobenius in characteristic $p$. The rule
\[
(E_x,t_x)\mapsto (E_x,t_x,\alpha_{\rm can}),
\]
where $\alpha_{\rm can}:E_x\mapsto E_x/{\rm can}(E_x)$ is the projection,
defines  rigid morphism $X_1(N)(\frac{p}{p+1})\longrightarrow\cW_\infty$,
and hence if $f$ is a modular form of weight $k$ and level $\Gamma$,
then the restriction $f\vert_{\cW_\infty}$ gives an overconvergent
$p$-adic modular form of weight $k$ and tame level $N$.

\subsection{Ordinary CM points}\label{subsubsec:A}

Let $K$ be an imaginary quadratic field with ring of integers $\cO_K$ equipped
with a cyclic ideal $\mathfrak{N}\subset\cO_K$ such that
\[
\cO_K/\mathfrak{N}\simeq\bZ/N\bZ.
\]
Fix an elliptic curve $A$ defined over the Hilbert class field $H$ of $K$ with
${\rm End}_H(A)\simeq\cO_K$ having good reduction at the primes above $p$,
and choose a $\Gamma_1(N)$-level structure $t_A\in A[\mathfrak{N}]$ and
a regular differential $\omega_A\in\Omega^1_{A/H}$. The identification
${\rm End}_H(A)=\cO_K$ is normalized so that $\lambda\in\cO_K$ acts as
\[
\lambda^*\omega=\lambda\omega\quad\textrm{for all $\omega\in\Omega_{A/H}^1$.}
\]

For every integer $c\geq 1$ prime to $Np$,
let $\cO_c=\bZ+c\cO_K$ be the order of $K$ of conductor $c$,
and denote by ${\rm Isog}_c^\mathfrak{N}(A)$ the set of 
elliptic curves $A'$ with CM by $\cO_c$ equipped with an isogeny $\varphi:A\longrightarrow A'$
satisfying ${\rm ker}(\varphi)\cap A[\mathfrak{N}]=\{0\}$.
\sk

The semigroup of projective rank one $\cO_c$-modules $\fa\subset\cO_c$ prime to
$\mathfrak{N}\cap\cO_c$ acts on ${\rm
Isog}_c^{\mathfrak{N}}(A)$ by the rule
\[
\fa*(\varphi:A\arrow
A')=\varphi_\fa\varphi:A\arrow A'\arrow
A'_\fa,
\]
where $A'_\fa:=A'/A'[\fa]$ and $\varphi_\fa:A'\longrightarrow A'_{\fa}$
is the natural projection. It is easily seen that this induces an action
of ${\rm Pic}(\cO_c)$ on ${\rm Isog}_c^{\mathfrak{N}}(A)$.
\sk

Throughout this paper, we shall assume that $p=\pp\overline{\pp}$ splits in $K$,
and let $\mathfrak{p}$ be the prime of $K$ above $p$ induced by our fixed embedding
$\overline{\bQ}_p\overset{\imath_p}\hookrightarrow\bC_p$.
Thus if $A'$ is an elliptic curve with CM by $\cO_c$ defined over the ring class field $H_c$ of $K$ of conductor $c$,
then $A'$ has ordinary reduction at $p$, and $A'[\pp]\subset A'[p]$ is the canonical subgroup.
In the following, we will let $\alpha_\pp'=\alpha_{\rm can}:A'\longrightarrow A'/A'[\pp]$
denote the projection.

%


\subsection{Generalized Heegner cycles}

For any $r>0$, let $W_r$ be the Kuga--Sato variety
over
\[
X_0:=X_1(Np)
\]
obtained as the canonical desingularization of the
$r$-fold self-product of the universal generalized elliptic curve over
$X_0$, and define
\begin{equation}\label{genKS}
X_r:=W_r\times A^r,
\end{equation}
where $A/H$ is the elliptic curve with CM by $\cO_K$ fixed in the preceding section. 
\sk

The variety $X_r$ is fibered over $X_0$, 
and the fiber over a non-cuspidal point
$x$ associated with a pair $(E_x,t_{x})$ 
is identified with $E_x^r\times A^r$.
Thus for every isogeny $\varphi:A\longrightarrow A'$
in ${\rm Isog}_c^\mathfrak{N}(A)$, we may consider the cycle
\[
\Upsilon_\varphi:=(\Gamma_\varphi^{\rm t})^r\;\subset\;(A'\times A)^r\;\subset\; X_r,
\]
where $\Gamma_\varphi^{\rm t}$ is the transpose of the graph of $\varphi$, and following
\cite[\S{2.3}]{bdp1} define the \emph{generalized Heegner cycle} associated with $\varphi$ by
\begin{equation}\label{genheeg}
\Delta_\varphi:=\epsilon_{X}\Upsilon_\varphi,
\end{equation}
where $\epsilon_{X}$ is the idempotent defined in [\emph{loc.cit.}, (2.1.1)]
(with $X_0$ in place of their curve $C=X_1(N)$).
By \cite[Prop.~2.7]{bdp1}, the cycles $\Delta_\varphi$ are homologically trivial; 
by abuse of notation, we shall still denote by $\Delta_\varphi$ the classes they define
in the Chow group ${\rm CH}^{r+1}(X_r)_0$ with rational coefficients.
For $r=0$, set
\[
\Delta_\varphi:=(A',t_{A'})-(\infty),
\]
where $t_{A'}\in A'[Np]$ is a $\Gamma_1(Np)$-level structure contained in $A'[\mathfrak{Np}]$,
and $\infty$ is the cusp $({\rm Tate}(q),\zeta_{Np})$. 




\section{A semistable non-crystalline setting}

This section is aimed at proving Theorem~\ref{thmbdp1A} below,
which extends the $p$-adic Gross--Zagier formula due Bertolini--Darmon--Prasanna~\cite{bdp1}
in the good reduction case to the semistable non-crystalline setting.

\subsection{$p$-adic Abel--Jacobi maps}\label{sec:p-adic-AJ}

Let $F$ be a finite unramified extension of $\bQ_p$,
denote by $\cO_F$ the ring of integers of $F$, and let $\kappa$ be the residue field.
The generalized Kuga--Sato variety $X_{r}$,
which was defined in $(\ref{genKS})$ as a scheme over $\bZ[1/Np]$,
has semistable reduction at $p$. In other words,
there exits a proper scheme $\mathcal{X}_r$ over $\cO_F$
with generic fiber $X_r\times_{\bZ[1/Np]}F$ and with
special fiber $\cX_r\times_{\cO_F}\kappa$
whose only singularities are divisors with normal crossings.
\sk

By the work of Hyodo--Kato \cite{hk},
attached to $X_r$ there are log-crystalline cohomology
groups $\hcris^j(\cX_r\times_{\cO_F}\kappa)$, which are $\cO_F$-modules of finite rank equipped with
a semilinear Frobenius automorphism $\Phi$
and a linear nilpotent monodromy operator $N$ satisfying
\begin{displaymath}
N\Phi=p\Phi N.
\end{displaymath}
Moreover, for each choice of a uniformizer of
$\cO_F$ there is a comparison isomorphism
\begin{displaymath}
\hcris^j(\cX_r\times_{\cO_F}\kappa)\otimes_{\cO_F}F
\simeq H_{\rm dR}^j(X_r/F)
\end{displaymath}
endowing the algebraic de Rham cohomology groups
$H_{\rm dR}^j(X_r/F)$ with the structure of
filtered $(\Phi,N)$-modules. In the following,
we shall restrict our attention to the middle degree cohomology, i.e.,
we set $j=2r+1$.
\sk

Let $G_F:={\rm Gal}(\overline{F}/F)$ be the absolute Galois group of $F$,
and consider the $p$-adic $G_F$-representation given by
\[
V_r:=H_{\textrm{\'et}}^{2r+1}(X_r\times_F\overline{F},\bQ_p).
\]
Applying Fontaine's functor $\mathbf{D}_{\rm st}$ to $V_r$ yields another filtered
$(\Phi,N)$-module associated to $X_r$.

\begin{thm}[Tsuji]\label{hk}
The $p$-adic $G_F$-representation $V_r$ is semistable,
and there is a natural
isomorphism
\begin{equation}\label{comp}
\mathbf{D}_{\rm st}(V_r)\simeq H_{\rm dR}^{2r+1}(X_r/F)\nonumber
\end{equation}
compatible with all structures. In particular,
the assignment $V\mapsto\mathbf{D}_{\rm st}(V)$ induces an isomorphism
$\Ext_{\rm st}(\bQ_p,V_r)\simeq\Ext_{{\rm
Mod}_F(\Phi,N)}(F,H_{\rm dR}^{2r+1}(X_r/F))$.
\end{thm}

Here, $\Ext_{\rm st}(\bQ_p,V_r)\simeq H_{\rm st}^1(F,V_r):={\rm ker}(H^1(F,V_r)
\longrightarrow H^1(F,V_r\otimes_{\bQ_p}\mathbf{B}_{\rm st}))$
is the group of extensions of the trivial representation $\bQ_p$ by $V_r$
in the category of semistable $p$-adic $G_F$-representations.
\sk

The idempotent $\epsilon_X$ used in the definition $(\ref{genheeg})$ of
the generalized Heegner cycles $\Delta_\varphi$ acts as a projector on the various cohomology
groups associated to the variety $X_r$. Let $V_r(r+1)$ be the $(r+1)$-st Tate twist of $V_r$,
and consider the \'etale Abel--Jacobi map
\[
{\rm AJ}_F^{\textrm{\'et}}:{\rm
CH}^{r+1}(X_r)_0(F)\arrow{\rm
Ext}_{{\rm Rep}_{G_F}}(\bQ_p,\epsilon_XV_r(r+1))=H^1(F,\epsilon_XV_r(r+1))
\]
constructed in \cite{nekovarCRM}.
By [\emph{loc.cit.}, Thm.~3.1$(ii)$],
the image of ${\rm AJ}_F^{\textrm{\'et}}$
lands in $H_{\rm st}^1(F,\epsilon_XV_r(r+1))$, 
and hence via the comparison isomorphism (\ref{comp})
it can be seen as taking values in the group
\[
\Ext_{{\rm
Mod}_F(\Phi,N)}(F,\epsilon_XH_{\rm dR}^{2r+1}(X_r/F)(r+1))
\]
of extensions of $F$ by the twist $\epsilon_XH_{\rm dR}^{2r+1}(X_r/F)(r+1)$
in the category of filtered $(\Phi,N)$-modules over $F$.
This group admits the following explicit description.

\begin{lem}\label{fil}
Set $H_r:=H_{\rm dR}^{2r+1}(X_r/F)$ and let $n=[F:\bQ_p]$. The assignment
\[
\{0\longrightarrow\epsilon_XH_r(r+1)\longrightarrow E\xrightarrow{\;\rho\;}F\longrightarrow 0\}
\quad\rightsquigarrow\quad\eta_E=\eta_E^{\rm hol}(1)-\eta_E^{\rm frob}(1),
\]
where $\eta_E^{\rm hol}:F\longrightarrow{\rm Fil}^0E$ (resp.
$\eta_E^{\rm frob}:F\longrightarrow E^{\Phi^n=1,N=0}$) is a section of
$\rho$ compatible with filtrations
(resp. with Frobenius and monodromy) yields an isomorphism
\begin{equation}\label{exp}
{\rm Ext}_{{\rm Mod}_F(\Phi,N)}(F,\epsilon_XH_r(r+1))
\simeq\epsilon_XH_r(r+1)/{\rm Fil}^0\epsilon_XH_r(r+1).\nonumber
\end{equation}
\end{lem}

\begin{proof}
See \cite[Lemma~2.1]{IS}, for example.
\end{proof}

Define the \emph{$p$-adic Abel--Jacobi map}
\begin{equation}\label{aj}
{\rm AJ}_F:{\rm CH}^{r+1}(X_r)_0(F)\arrow\epsilon_XH_{\rm dR}^{2r+1}(X_r/F)(r+1)/{\rm
Fil}^{0}\epsilon_XH_{\rm dR}^{2r+1}(X_r/F)(r+1)
\end{equation}
to be the composite of ${\rm AJ}_F^{\textrm{\'et}}$
with the isomorphisms of Theorem~\ref{hk} and Lemma~\ref{exp}.
Since the filtered pieces ${\rm Fil}^{1}\epsilon_XH^{2r+1}_{\rm dR}(X_r/F)(r)$
and ${\rm Fil}^0\epsilon_XH^{2r+1}_{\rm dR}(X_r/F)(r+1)$ are exact annihilators
under the Poincar\'e duality
\begin{equation}\label{PD-dR}
\epsilon_XH^{2r+1}_{\rm dR}(X_r/F)(r)\times\epsilon_XH^{2r+1}_{\rm dR}(X_r/F)(r+1)\arrow F,\nonumber
\end{equation}
the target of ${\rm AJ}_F$ may be identified
with the linear dual $({\rm Fil}^{r+1}\epsilon_XH^{2r+1}_{\rm dR}(X_r/F))^\vee$.
\sk

Recall the coherent sheaf of $\cO_X$-modules $\ccL_r={\rm Sym}^r\ccL$ on $X$
introduced in Section~\ref{subsubsec:mf&dR}, and set
\[
\ccL_{r,r}:=\ccL_r\otimes{\rm Sym}^rH_{\rm dR}^1(A).
\]
With the trivial extension of the Gauss--Manin connection $\nabla$ on $\ccL_r$
to $\ccL_{r,r}$, consider the complex
\[
\ccL_{r,r}^\bullet:\ccL_{r,r}
\xrightarrow{\;\nabla\;}\ccL_{r,r}\otimes\Omega_X^1(\log Z_X),
\]
and define $H_{\rm par}^1(X,\ccL_{r,r},\nabla)$ as in $(\ref{def:Hpar})$.
By \cite[Prop.~2.4]{bdp1}, we then have
\begin{equation}
\epsilon_{X}H^{2r+1}_{\rm dR}(X_r/F)\simeq H_{\rm par}^1(X,\ccL_{r,r},\nabla)
=H^1_{\rm par}(X,\ccL_r,\nabla)\otimes{\rm Sym}^rH^1_{\rm dR}(A/F)\nonumber
\end{equation}
and
\begin{equation}\label{2.4}
{\rm Fil}^{r+1}\epsilon_XH^{2r+1}_{\rm dR}(X_r/F)
\simeq H^0(X,\underline{\omega}^{\otimes r}\otimes\Omega_{X}^1)\otimes{\rm Sym}^rH_{\rm dR}^1(A/F).\nonumber
\end{equation}

As a result of these identifications, we shall view the $p$-adic Abel--Jacobi map $(\ref{aj})$
as a map
\begin{equation}\label{p-AJ}
{\rm AJ}_F:{\rm CH}^{r+1}(X_r)_0(F)
\arrow(H^0(X,\underline{\omega}^{\otimes r}\otimes\Omega_{X}^1)\otimes{\rm Sym}^rH_{\rm dR}^1(A/F))^\vee.
\end{equation}
Moreover, if $\Delta=\epsilon_X\Delta\in{\rm CH}^{r+1}(X_r)_0(F)$
is the class of a cycle in the image of the idempotent $\epsilon_X$
supported on the fiber of $X_r\longrightarrow X$ over a point $P\in X(F)$,
we see that ${\rm AJ}_F(\Delta)$ may be computed using the following recipe.
Consider the commutative diagram with Cartesian squares:
\begin{displaymath}
\xymatrix{
0\ar[r]&H_{\rm par}^1(X,\ccL_{r,r},\nabla)(r+1) \ar[r]\ar@{=}[d]& D_\Delta
\ar[r]\ar[d] & F \ar[d]^{{\rm cl}_\Delta}\ar[r]&0\\
0\ar[r]&H_{\rm par}^1(X,\ccL_{r,r},\nabla)(r+1) \ar[r]&
H^1_{\rm par}(X\smallsetminus P,\ccL_{r,r},\nabla)(r+1) \ar[r]& \ccL_{r,r}(P)(r)\ar[r]&0,
}
\end{displaymath}
where the rightmost vertical map is defined by
sending $1\in F$ to the cycle class ${\rm cl}_P(\Delta)$. 
Then ${\rm AJ}_F(\Delta)$ is given by the linear functional
\begin{equation}\label{computeAJ}
{\rm AJ}_F(\Delta)=\langle-,\eta_\Delta\rangle,\nonumber
\end{equation}
where $\eta_\Delta=\eta_\Delta^{\rm hol}-\eta_\Delta^{\rm frob}:=\mathbf{\eta}_{D_\Delta}^{\rm hol}(1)-\mathbf{\eta}_{D_\Delta}^{\rm frob}(1)$
is the ``tangent vector'' associated as in Lemma~\ref{fil} to
the extension $D_\Delta$ as filtered $(\Phi,N)$-modules, and
\begin{equation}\label{PDparr}
\langle,\rangle:H^1_{\rm par}(X,\ccL_{r,r},\nabla)(r)\times H^1_{\rm par}(X,\ccL_{r,r},\nabla)(r+1)\longrightarrow F
\end{equation}
is the Poincar\'e duality.

\subsection{Rigid cohomology}\label{sec:p-adic-dR}

Recall the rigid spaces
$\cZ_\infty\subset\cW_\infty$, $\cZ_0\subset\cW_0$
introduced in Section~\ref{subsubsec:X}.
Fix a collection of points $\{P_1,\dots,P_t\}$ of $X(F)$ contained in $\cZ_\infty$,
containing all the cusps of $\cZ_\infty$, and such that ${\rm red}_p(P_i)\neq{\rm red}_p(P_j)$ for $i\ne j$.
Let $w_p$ be the automorphism of $X$ defined in terms of moduli by
\begin{equation}\label{wp}
w_p(E,t,\alpha)=(\alpha(E),\alpha(t),\alpha^\vee),
\end{equation}
where $\alpha^\vee$ is the isogeny dual to $\alpha$,
and set $P_j^*:=w_pP_j$. Then the points $P_j^*$ factor through $\cZ_0$,
and the set
\[
S:=\{P_1,\dots,P_t,P_1^*,\dots,P_t^*\}
\]
contains all the cusps of $X$.
Since the points $Q\in S$ reduce to smooth points
$\bar{Q}$ in the special fiber, the spaces
$D(Q):={\rm red}_p^{-1}(\bar{Q})$ are conformal to
the open disc in $D(0;1)$ in $\bC_p$.
Fix isomorphisms $h_Q:D(Q)\longrightarrow D(0;1)$
mapping the point $Q$ to $0$, and for a
collection of real numbers $r_Q<1$ consider the annuli
\begin{equation}\label{eq:Vj}
\cV_Q:=\{x\in D(Q)\;\colon\;r_Q<|h_Q(x)|_p<1\}.
\end{equation}

Denote by $\ccL_{r,r}^{\rm rig}$ the sheaf for the rigid analytic topology on $X(\bC_p)$
defined by the algebraic vector bundle $\ccL_{r,r}$.
If $\cV\subset X(\bC_p)$ is a connected wide open contained in $Y(\bC_p)$,
the Gauss--Manin connection yields a connection
\[
\nabla:\ccL_{r,r}^{\rm rig}\vert_\cV\arrow\ccL_{r,r}^{\rm rig}\vert_{\cV}\otimes\Omega_\cV^1,
\]
and similarly as in (\ref{def:HdR})
we define the \emph{$i$-th de Rham cohomology} of $\cV$ attached
to $\ccL^{\rm rig}_{r,r}$ by
\[
H^i(\ccL_{r,r}^{\bullet}\vert_\cV)=H_{\rm dR}^i(\cV,\ccL^{\rm rig}_{r,r},\nabla)
:=\mathbb{H}^i(\ccL_{r,r}^{\rm rig}\vert_\cV\xrightarrow{\;\nabla\;}\ccL_{r,r}^{\rm
rig}\vert_{\cV}\otimes\Omega_\cV^1).
\]
In particular, if $\cV$ is a basic wide open, then
\[
H^1(\ccL_{r,r}^\bullet\vert_\cV)\simeq\frac{\ccL_{r,r}^{\rm rig}(\cV)\otimes\Omega_\cV^1}
{\nabla\ccL_{r,r}^{\rm rig}(\cV)},
\]
and $H^0(\ccL_{r,r}^\bullet\vert_{\cV})\simeq\ccL_{r,r}^{\rm rig}(\cV)^{\nabla=0}$
is the space of horizontal sections of $\ccL_{r,r}^{\rm
rig}$ over $\cV$. For $r=0$, we set
\[
H^1(\ccL_{r,r}^\bullet\vert_{\cV})=H^1(\cV):=\Omega_{\cV}^1/d\cO_{\cV},\quad\quad
H^0(\ccL_{r,r}^\bullet\vert_{\cV})=H^0(\cV):=\cO_\cV^{d=0},
\]
where $d:\cO_\cV\longrightarrow\Omega^1_\cV$ is the differentiation map.
\sk

In terms of the admissible cover of $X(\bC_p)$ 
by basic wide opens described in Section~\ref{subsubsec:X},
the classes in $H^1_{\rm dR}(X(\bC_p),\ccL_{r,r}^{\rm rig},\nabla)$
may be represented by hypercocycles $(\omega_0,\omega_\infty;f_{\cW})$,
where $\omega_0$ and $\omega_\infty$ are $\ccL_{r,r}^{\rm rig}$-valued differentials
on $\cW_0$ and $\cW_\infty$, respectively,
and $f_{\cW}\in\ccL_{r,r}^{\rm rig}(\cW)$ is such that
$(\omega_\infty-\omega_0)\vert_{\cW}=\nabla f_{\cW}$; and two
hypercocycles represent the same class if
their difference is of the form $(\nabla f_0,\nabla
f_\infty;(f_\infty-f_0)\vert_\cW)$ for some $f_0\in\ccL_{r,r}^{\rm rig}(\cW_0)$
and $f_\infty\in\ccL_{r,r}^{\rm rig}(\cW_\infty)$.
\sk

If $\cV$ is a wide open annulus,
associated with an orientation of $\cV$ there is a \emph{$p$-adic
annular residue}
\begin{equation}\label{res0}
{\rm res}_\cV:\Omega_\cV^1\arrow\bC_p
\end{equation}
defined by expanding $\omega=\sum_na_nT^n\frac{dT}{T}\in\Omega_\cV^1$
with respect to a fixed uniformizing parameter $T$ of compatible with the orientation,
and setting ${\rm res}_\cV(\omega):=a_{0}$ (see \cite[Lemma~2.1]{rlc}).
Combined with the natural pairing
\[
\langle,\rangle:\ccL_{r,r}^{\rm rig}(\cV)\times\ccL_{r,r}^{\rm
rig}(\cV)\otimes\Omega_\cV^1\arrow\Omega_\cV^1
\]
induced by the Poincar\'e duality $(\ref{PDL})$ on $\ccL_{r}$
(extended to $\ccL_{r,r}$ 
in the obvious manner),
we obtain a higher $p$-adic annular residue map
\begin{equation}\label{bdp-residue}
{\rm Res}_\cV:\ccL_{r,r}^{\rm
rig}(\cV)\otimes\Omega_\cV^1\arrow\ccL_{r,r}^{\rm
rig}(\cV)^\vee
\end{equation}
by setting
\[
{\rm Res}_{\cV}(\omega)(\alpha)
={\rm res}_{\cV}\langle\alpha,\omega\rangle
\]
for every  $\ccL^{\rm rig}_{r,r}$-valued differential $\omega$ on $\cV$
and every section $\alpha\in\ccL^{\rm rig}_{r,r}(\cV)$.
Since ${\rm res}_\cV$ clearly descends to a map $H^1(\cV)=\Omega_{\cV}^1/d\cO_{\cV}\longrightarrow\bC_p$,
by composing ${\rm Res}_\cV$ with the projection $\ccL_{r,r}^{\rm rig}(\cV)^\vee\longrightarrow
H^0(\ccL_{r,r}^\bullet\vert_{\cV})^\vee$ 
it is easily seen from the Leibniz rule that we obtain a well-defined map
\begin{equation}\label{higherRes}
{\rm Res}_\cV:H^1(\ccL_{r,r}^\bullet\vert_{\cV})\arrow
H^0(\ccL_{r,r}^\bullet\vert_{\cV})^\vee.
\end{equation}

If $\cV_Q\subset D(Q)$ is the annulus attached to
a non-cuspidal point $Q\in S$, it will be convenient,
following the discussion after \cite[Cor.~3.7]{bdp1},
to view ${\rm Res}_{\cV_Q}$ as taking values on the fiber $\ccL_{r,r}(Q)$, using
the sequence of identifications
\begin{equation}\label{eq:ident}
H^0(\ccL_{r,r}^\bullet\vert_{\cV_Q})^\vee
=(H^0(D(Q),\ccL_{r,r})^{\nabla=0})^\vee =\ccL_{r,r}(Q)^\vee=\ccL_{r,r}(Q)
\end{equation}
arising from ``analytic continuation'', the choice of
an ``initial condition'', and the self-duality of $\ccL_{r,r}(Q)$, respectively.
(See \emph{loc.cit.} for the case of a cusp $Q\in S$.)
\sk

For a supersingular annulus $\A_x$, the vector space
$H^0(\ccL_{r,r}^\bullet\vert_{\A_x})$ is equipped
with a pairing $\langle,\rangle_{\A_x}$, arising from
an identification (similar to $(\ref{eq:ident})$)
with the de Rham cohomology of a supersingular elliptic curve
in characteristic $p$ corresponding to $x\in SS$.
Moreover, since $H^0(\A_x)\simeq\bC_p$,
the residue map (\ref{res0}) yields an isomorphism
${\rm res}_{\A_x}:H^1(\A_x)\longrightarrow H^0(\A_x)$,
and using a trivialization of $\ccL_{r,r}^{\rm rig}\vert_{\A_x}$
it may be extended to an isomorphism
\begin{equation}\label{Resr}
{\rm Res}_{\A_x}:H^1(\ccL_{r,r}^\bullet\vert_{\A_x})
\simeq H^0(\ccL_{r,r}^\bullet\vert_{\A_x})
\end{equation}
(see \cite[Prop.~7.1]{pSI}). It is then easily checked that $(\ref{higherRes})$
and $(\ref{Resr})$ correspond to each other under the identification
$H^0(\ccL_{r,r}^\bullet\vert_{\A_x})^\vee=H^0(\ccL_{r,r}^\bullet\vert_{\A_x})$
defined by $\langle,\rangle_{\A_x}$.
\sk

Let $S$ be a set of points as introduced above, 
and define
\begin{align*}
\cW_\infty^\sharp&:=\cZ_\infty\smallsetminus
\bigcup_{Q\in S\cap\cZ_\infty}D(Q)\smallsetminus\cV_Q,\quad\quad
\cU:={\cW}_\infty^\sharp\cup{\cW}_0^\sharp,
\end{align*}
where $\cW_0^\sharp:=w_p\cW_\infty^\sharp$, and $U:=X\smallsetminus S$.
The restriction of an $\ccL_{r,r}$-valued differential on $X$ which is regular on $U$
defines a section of $\ccL_{r,r}^{\rm rig}\otimes\Omega_X^1$ over
$\cU$. As argued in the proof of \cite[Prop.~7.2]{pSI}, this yields an isomorphism
\begin{equation}\label{alg=anal}
H_{\rm dR}^1(U,\ccL_{r,r},\nabla)\simeq H^1(\ccL_{r,r}^\bullet\vert_{\cU})\nonumber
\end{equation}
between algebraic and rigid de Rham cohomology.


\begin{prop}\label{prop:poincare}
Let the notations be as before.
\begin{enumerate}
\item{} A class $\kappa\in H^1(\ccL_{r,r}^\bullet\vert_{\cU})$
belongs to the image of $H_{\rm par}^1(X,\ccL_{r,r},\nabla)$
under restriction
\[
H_{\rm par}^1(X,\ccL_{r,r},\nabla)\arrow H_{\rm dR}^1(U,\ccL_{r,r},\nabla)\simeq H^1(\ccL_{r,r}^\bullet\vert_{\cU})
\]
if and only if ${\rm Res}_{\cV_Q}(\kappa)=0$ for all $Q\in S$.
\item{}
Let $V$ be such that $\{U,V\}$ is an admissible covering of $X$.
If $\kappa_\omega$, $\kappa_\eta\in H_{\rm par}^1(X,\ccL_{r,r},\nabla)$ are represented by the hypercocycles
$(\omega_U,\omega_V;\omega_{U\cap V})$, $(\eta_U,\eta_V;\eta_{U\cap V})$ respectively,
with respect to this covering, then the value
$\langle\kappa_\omega,\kappa_\eta\rangle$ under the Poincar\'e duality $(\ref{PDparr})$
is given by
\begin{equation}\label{eq:poincare}
\langle\kappa_\omega,\kappa_\eta\rangle=\sum_{Q\in S}{\rm
res}_{\cV_Q}\langle F_{\omega,Q},\eta_U\rangle,\nonumber
\end{equation}
where $F_{\omega,Q}$ is any local primitive of $\omega_U$ on $\cV_Q$,
i.e., such that $\nabla F_{\omega,Q}=\omega_U\vert_{\cV_Q}$.
\end{enumerate}
\end{prop}

\begin{proof}
The first assertion follows from the same argument as in \cite[Prop.~3.8]{bdp1},
and the second is \cite[Lemma~7.1]{pSI}.
%
\end{proof}

\subsection{Coleman's $p$-adic integration}\label{sec:fmandCol}

In this section, we give an explicit description of the
filtered $(\Phi,N)$-module structure on $H_{\rm par}^1(X,\ccL_{r,r},\nabla)$,
following the work of Coleman--Iovita \cite{CI2}.
We state the results for $\ccL_r$,
leaving their trivial extension to
$\ccL_{r,r}=\ccL_r\otimes{\rm Sym}^rH_{\rm dR}^1(A)$
to the reader.
\sk

As recalled in Section~\ref{sec:p-adic}, for every pair $(E_x,t_x)$
corresponding to a point $x\in X_1(N)(\frac{p}{p+1})$
there is a canonical $p$-isogeny $\alpha_{\rm can}:E_x\mapsto E_x/{\rm can}(E_x)$,
where ${\rm can}(E_x)\subset E_x[p]$ is the canonical subgroup.
The 
map $V:X_1(N)(\frac{1}{p+1})\longrightarrow X_1(N)(\frac{p}{p+1})$
defined in terms of moduli by
\begin{equation}\label{V}
V(E_x,t_x)=(\alpha_{\rm can}(E_x),\alpha_{\rm can}(t_x))
\end{equation}
is then a lift of the absolute Frobenius on $X_1(N)_{\mathbf{F}_p}$.
Letting $s_1:X_1(N)(\frac{p}{p+1})\longrightarrow\cW_\infty$
be defined by $(E_x,t_x)\mapsto(E_x,t_x,{\rm can}(E_x))$, and letting
$\cW_\infty'\subset\cW_\infty$ be the image of $X_1(N)(\frac{1}{p+1})$ under $s_1$,
the map $\phi_\infty$ defined by the commutativity of the diagram
\begin{displaymath}
\xymatrix{
\cW_\infty'
\ar[r]^{\phi_\infty}\ar[d]^{\pi_1}&\cW_\infty\\
X_1(N)(\frac{1}{p+1})\ar[r]^V&X_1(N)(\frac{p}{p+1}),\ar[u]^{s_1}
}
\end{displaymath}
is therefore a lift of the absolute Frobenius on $X_{\mathbf{F}_p}$.
\sk

As explained in \cite[p.41]{pSI} (see also the more detailed discussion in \cite[p.218]{coc}),
the canonical subgroup yields a horizontal morphism
${\rm Fr}_\infty:\phi_\infty^*\ccL_r\longrightarrow\ccL_r\vert_{\cW_\infty'}$. Define
the \emph{Frobenius endomorphism} ${\Phi}_\infty$
on $H^1(\ccL_r^\bullet\vert_{\cW_\infty})$ by the composite map
\[
H^1(\ccL_r^\bullet\vert_{\cW_\infty})\simeq
\frac{\ccL_r^{\rm rig}\otimes\Omega^1_{\cW_\infty}(\cW_\infty)}{\nabla\ccL_r^{\rm rig}(\cW_\infty)}
\xrightarrow{({\rm Fr}_\infty\otimes{\rm id})\phi_\infty^*}
\frac{\ccL_r^{\rm rig}\otimes\Omega^1_{\cW_\infty}(\cW_\infty')}{\nabla\ccL_r^{\rm rig}(\cW_\infty')}
\simeq
H^1(\ccL_r^\bullet\vert_{\cW_\infty}),
\]
where the last isomorphism is given by restriction (see \cite[Prop.~10.3]{pSI}).
Setting $\cW_0':=w_p\cW_\infty'\subset\cW_0=w_p\cW_\infty$
and $\phi_0:=w_p^{-1}\phi_\infty w_p$,
where $w_p$ is the automorphism of $X$ given by $(\ref{wp})$,
we similarly define a Frobenius endomorphism ${\Phi}_0$ of $H^1(\ccL_r^\bullet\vert_{\cW_0})$.

\begin{thm}[Coleman]\label{coleman-primitives}
Let $f=q+\sum_{n=2}^\infty a_n(f)q^n\in S_{r+2}(\Gamma_0(Np))$ 
be a $p$-new
eigenform of weight $r+2\geq 2$, and let $\omega_f\in H^0(X,\underline{\omega}^r\otimes\Omega_X^1)\subset H_{\rm
par}^1(X,\ccL_r,\nabla)$ be the associated differential.
Then for each $\star\in\{\infty,0\}$ there exists a
locally analytic section $F_{f,\star}$ of $\ccL_r$ on $\cW_\star$ such that
\begin{itemize}
\item[$(i)$]{} $\nabla F_{f,\star}=\omega_f\vert_{\cW_\star}$; and
\item[$(ii)$]{} $F_{f,\star}-\frac{a_p(f)}{p^{r+1}}\phi_\star^*F_{f,\star}$
is rigid analytic on $\cW_\star'$.
\end{itemize}
Moreover, $F_{f,\star}$ is unique modulo $H^0(\ccL_r^\bullet\vert_{\cW_\star})$.
\end{thm}

\begin{proof}
This follows from the discussion in \cite[\S{11}]{pSI}.
By [\emph{loc.cit}, Lemma~11.1] we have $\Phi_\infty=pU_p$
on the image of $S_{r+2}(X)^{p-{\rm new}}$ in
$H^1(\ccL_r^\bullet\vert_{\cW_\infty})$. Since
$U_p^2=p^r\langle p\rangle$ on the former space 
and we have the relations $U_p\omega_f=a_p(f)\omega_f$ and
$\langle p\rangle\omega_f=\omega_f$ by hypothesis,
it follows that the polynomial
\begin{equation}
P(T)=1-\frac{a_p(f)}{p^{r+1}}T\nonumber
\end{equation}
is such that
$P(\Phi_\infty)([\omega_f\vert_{\cW_\infty}])=0$,
and hence also $P(\Phi_0)([\omega_f\vert_{\cW_0}])=0$.
The result thus follows from \cite[Thm.~10.1]{pSI}.
\end{proof}

A locally analytic section $F_{f,\star}$ as in Theorem~\ref{coleman-primitives}
is called a \emph{Coleman primitive} of $f$ on $\cW_\star$.


\begin{rem}\label{remD}
For $r>0$, the spaces $H^0(\ccL_r^\bullet\vert_{\cW_\star})$ are trivial,
and so the Coleman primitives $F_{f,\star}$ are unique.
On the other hand, for $r=0$ we have $H^0(\ccL_r^\bullet\vert_{\cW_\star})\simeq\bC_p$,
and so the $F_{f,\star}$ are unique modulo a global constant on $\cW_\star$.
\end{rem}

\subsection{Frobenius and monodromy}

Denote by $\iota$ the inclusion of any
rigid subspace of $X$ into $X$.
Associated with the exact sequence of complexes of sheaves on $X$
\[
0\longrightarrow\ccL_r^\bullet\longrightarrow
\iota^*(\ccL_r^\bullet\vert_{\cW_0})\oplus\iota^*(\ccL_r^\bullet\vert_{\cW_\infty})
\xrightarrow{\rho_\infty-\rho_0}
\iota^*(\ccL_r^\bullet\vert_{\cW})\longrightarrow 0,
\]
there is a Mayer--Vietoris long exact sequence
\begin{align}
\cdots&\longrightarrow H^0_{\rm par}(\ccL_r^\bullet\vert_{\cW_0\sqcup\cW_\infty})
\xrightarrow{\;\;\beta^0\;\;}
H^0_{\rm par}(\ccL_r^\bullet\vert_{\cW})\xrightarrow{\;\;\delta\;\;}
H^1_{\rm par}(X,\ccL_r^{\rm rig},\nabla)\longrightarrow\nonumber\\
&\longrightarrow H^1_{\rm par}(\ccL_r^\bullet\vert_{\cW_0\sqcup\cW_\infty})
\xrightarrow{\;\;\beta^1\;\;} H^1_{\rm par}(\ccL_r^\bullet\vert_{\cW})
\longrightarrow\cdots\nonumber
\end{align}
in hypercohomology.
By \cite[\S{10}]{pSI} and the discussion in the preceding section,
each of the non-central spaces in the resulting short exact sequence
\begin{equation}\label{MV}
0\longrightarrow
\frac{H^0_{\rm par}(\ccL_{r}^\bullet\vert_{\cW})}{\beta^0(H_{\rm
par}^0(\ccL_{r}^\bullet\vert_{\cW_0\sqcup\cW_\infty}))}
\xrightarrow{\;\;\;\delta\;\;\;}H_{\rm par}^1(X,\ccL_r,\nabla)
\longrightarrow H^1_{\rm par}(\ccL_r^\bullet\vert_{\cW_0\sqcup\cW_\infty})^{\beta^1=0}
\longrightarrow 0
\end{equation}
is equipped with a Frobenius endomorphism. Therefore, to define a Frobenius action on
$H^1_{\rm par}(X,\ccL_r,\nabla)$ it suffices to construct a splitting of $(\ref{MV})$.
\sk

As shown in \cite[\S{A.5}]{pSI}, this may be obtained as follows. Assume that $\kappa\in H_{\rm
par}^1(X,\ccL_r,\nabla)$ is represented by the hypercocycle
$(\omega_0,\omega_\infty;f_\cW)$ with respect
to the covering $\{\cW_0,\cW_\infty\}$ of $X$. Since $\cW=\bigcup_{x\in SS}\A_x$
is the union of the supersingular annuli,
we may write $f_\cW=\{f_x\}_{x\in SS}$ with $f_x\in\ccL_r^{\rm rig}(\A_x)$.
The assignment
\begin{equation}\label{s}
\A_x\longmapsto F_{\omega_\infty}\vert_{\A_x}-F_{\omega_0}
\vert_{\A_x}-f_x,
\end{equation}
where $F_{\omega_\star}$ is a Coleman
primitive of $\omega_\star$ on $\cW_\star$,
defines a horizontal section of $\ccL_r^{\rm rig}$
on $\cW$, and its image modulo $\beta^0(H_{\rm
par}^0(\ccL_{r}^\bullet\vert_{\cW_0\sqcup\cW_\infty}))$
is independent of the chosen $F_{\omega_\star}$ (see Remark~\ref{remD}).
It is easily checked that $s\delta={\rm id}$,
and hence we may define a Frobenius operator $\Phi$ on $H_{\rm
par}^1(X,\ccL_r,\nabla)$ by requiring that its
action be compatible with the resulting splitting of $(\ref{MV})$.
\sk

On the other hand, define the monodromy operator $N$
on $H_{\rm par}^1(X,\ccL_r,\nabla)$ by the composite map
\[
H_{\rm par}^1(X,\ccL_r,\nabla)\longrightarrow
H^1(\ccL_{r}^\bullet\vert_{\cW})\xrightarrow{\bigoplus_{x\in SS}{\rm
Res}_{\A_x}}H^0(\ccL_{r}^\bullet\vert_{\cW})\xrightarrow{\;\;\delta\;\;}
H_{\rm par}^1(X,\ccL_r,\nabla),
\]
where ${\rm Res}_{\A_x}$ are the $p$-adic residue maps $(\ref{Resr})$.

\begin{lem}\label{N=0}
Let $\kappa\in H_{\rm par}^1(X,\ccL_r,\nabla)$. Then:
\begin{itemize}
\item[$(i)$] For $r>0$, $N(\kappa)=0\;\Longleftrightarrow\;{\rm Res}_{\A_x}(\kappa)=0$ for all $x\in SS$;
\item[$(ii)$] For $r=0$, $N(\kappa)=0\;\Longleftrightarrow\;$ there is $C\in\bC_p$ such that
${\rm res}_{\A_x}(\kappa)=C$ for all $x\in SS$.
\end{itemize}
\end{lem}

\begin{proof}
This follows immediately from the exact sequence $(\ref{MV})$ and
the determination of the spaces $H^0(\ccL_r^\bullet\vert_{\cW_\star})$
recalled in Remark~\ref{remD}.
\end{proof}

By the main result of \cite{CI2}, the operators $\Phi$ and $N$ on
$H_{\rm par}^1(X,\ccL_r,\nabla)$ defined above
agree with the corresponding structures
deduced from the comparison isomorphism of Theorem~\ref{hk}.

\subsection{$p$-adic Gross--Zagier formula}\label{sec:computations}

Fix a finite extension $F/\bQ_p$ containing the image of the Hilbert class field
$H$ of $K$ under our fixed embedding $\overline{\bQ}\overset{\imath_p}\hookrightarrow\bC_p$,
and let $c\geq 1$ be an integer prime to $Np$.

\begin{prop}\label{prop:AJ-Coleman}
Let $f=q+\sum_{n=2}^{\infty}a_n(f)q^n\in S_{r+2}(\Gamma_0(Np))$ be a $p$-new eigenform of weight $r+2\geq 2$. 
Let $\varphi:A\longrightarrow A'$ be an isogeny in ${\rm Isog}_c^{\mathfrak{N}}(A)$, let
$P_{A'}\in X(F)$ be the point defined by $(A',t_{A'})$, and
let $\Delta_\varphi$ 
be generalized Heegner cycle associated to $\varphi$.
Then for all $\alpha\in{\rm Sym}^rH_{\rm dR}^1(A/F)$, we have
\begin{equation}\label{Delta-}
{\rm AJ}_F(\Delta_\varphi)(\omega_{f^{}}\wedge\alpha)
=\langle F_{f,\infty}(P_{A'})\wedge\alpha,{\rm cl}_{P_{A'}}(\Delta_\varphi)\rangle,\nonumber
\end{equation}
where $F_{f,\infty}$ is the Coleman primitive of $\omega_f\in H^0(X,\underline{\omega}^{\otimes r}\otimes\Omega^1_X)$
on $\cW_\infty$ (vanishing at $\infty$ if $r=0$), and the pairing on the right-hand side is the natural one
on $\ccL_{r,r}(P_{A'})$.
\end{prop}

\begin{proof}
Following the recipe described at the end of Section~\ref{sec:p-adic-AJ}, we have
\begin{equation}\label{eq:first}
{\rm AJ}_F(\Delta_\varphi)(\omega_{f^{}}\wedge\alpha)
=\langle\omega_{f^{}}\wedge\alpha,\eta_\Delta^{\rm hol}-\eta_\Delta^{\rm frob}\rangle,
\end{equation}
where:
\begin{itemize}
\item{} $\eta_\Delta^{\rm hol}$ is a cohomology class represented by a section
(still denoted $\eta_{\Delta}^{\rm hol}$) of $\ccL_{r,r}\otimes\Omega_X^1(\log Z_X)$ over $U$
having residue $0$ at the cusps, and with a simple pole at $P_{A'}$ with residue ${\rm cl}_{P_{A'}}(\Delta_\varphi)$;
\item{} $\eta_\Delta^{\rm frob}$ is section of $\ccL_{r,r}^{\rm rig}\otimes\Omega_X^1$ over $\cU$
having the same residues as $\eta_\Delta^{\rm hol}$, and satisfying $N(\eta_\Delta^{\rm frob})=0$
and
\begin{equation}\label{frob}
\Phi\eta_\Delta^{\rm frob}=\eta_\Delta^{\rm frob}+\nabla G, 
\end{equation}
for some rigid section $G$ of $\ccL_{r,r}^{\rm rig}$
on a strict neighborhood of $(\cZ_0\cap\cW_0^\sharp)\cup(\cZ_\infty\cap\cW_\infty^\sharp)$ in $\cU$.
\end{itemize}
By the formula for the Poincar\'e pairing in Proposition~\ref{prop:poincare},
equation $(\ref{eq:first})$ may be rewritten as
\begin{align}\label{AJdelta}
{\rm AJ}_F(\Delta_\varphi)(\omega_{f}\wedge\alpha)
&=\sum_{Q\in S}{\rm res}_{\cV_Q}\langle F_{f^{},Q}
\wedge\alpha,\eta_\Delta^{\rm hol}-\eta_\Delta^{\rm frob}\rangle,
\end{align}
where $F_{f^{},Q}\in\ccL_r^{}(\cV_Q)$ is an arbitrary local primitive of $\omega_{f^{}}$ on the annulus $\cV_Q$.
(Note that here we are using the fact that the connection $\nabla$ on
$\ccL_{r,r}=\ccL_r\otimes{\rm Sym}^rH_{\rm dR}^1(A/F)$ is defined from
the Gauss--Manin connection on $\ccL_r$ by extending it trivially on the second factor.)

If $F_{f,\infty}$ is a Coleman primitive of $\omega_f$ on $\cW_\infty$,
then $F_{f,\infty}^{[p]}:=F_{f,\infty}-\frac{a_p(f)}{p^{r+1}}\phi_\infty^*F_{f,\infty}$
is rigid analytic on $\cW_\infty'\subset\cW_\infty$ by Theorem~\ref{coleman-primitives},
and hence
\begin{equation}\label{+Ax}
\sum_{Q\in S\cap\cW_\infty}{\rm res}_{\cV_Q}\langle
F_{f,\infty}^{[p]}\wedge\alpha,\eta_\Delta^{\rm frob}\rangle
+\sum_{x\in SS}{\rm res}_{\A_x}\langle
F_{f,\infty}^{[p]}\wedge\alpha,\eta_{\Delta}^{\rm frob}\rangle=0
\end{equation}
by the Residue Theorem (see \cite[Thm.~3.8]{bdp1}).
Since $N(\eta_\Delta^{\rm frob})=0$,
Lemma~\ref{N=0} implies that we can write $\eta_\Delta^{\rm frob}=\nabla G_x$ for
some rigid section $G_x\in\ccL_{r,r}^{\rm rig}(\A_x)$ on each supersingular
annulus $\A_x$, and hence
\begin{align*}
d\langle F_{f,\infty}^{[p]}\wedge\alpha,G_x\rangle
&=\langle\nabla F_{f,\infty}^{[p]}\wedge\alpha,G_x\rangle+
\langle F_{f,\infty}^{[p]}\wedge\alpha,\eta_\Delta^{\rm frob}\rangle.
\end{align*}
In particular, the right-hand side in the last equality has residue $0$, and hence
\begin{equation}\label{Ax}
{\rm res}_{\A_x}\langle F_{f,\infty}^{[p]}\wedge\alpha,\eta_\Delta^{\rm frob}\rangle
=-{\rm res}_{\A_x}\langle\nabla F_{f,\infty}^{[p]}\wedge\alpha,G_x\rangle,
\end{equation}
Plugging $(\ref{Ax})$ into $(\ref{+Ax})$, we arrive at
\begin{equation}\label{00}
\sum_{Q\in S\cap\cW_\infty}{\rm res}_{\cV_Q}\langle
F_{f,\infty}^{[p]}\wedge\alpha,\eta_\Delta^{\rm frob}\rangle
-\sum_{x\in SS}{\rm res}_{\A_x}\langle
\nabla F_{f,\infty}^{[p]}\wedge\alpha,\eta_{\Delta}^{\rm frob}\rangle=0.
\end{equation}
An entirely parallel reasoning with $\cW_0$ in place of $\cW_\infty$ yields a proof of the equality
\begin{equation}\label{0}
\sum_{Q\in S\cap\cW_0}{\rm res}_{\cV_Q}\langle
F_{f,0}^{[p]}\wedge\alpha,\eta_\Delta^{\rm frob}\rangle
+\sum_{x\in SS}{\rm res}_{\A_x}\langle
\nabla F_{f,\infty}^{[p]}\wedge\alpha,\eta_{\Delta}^{\rm frob}\rangle=0,
\end{equation}
where $F_{f,0}$ is a Coleman primitive of $\omega_f$ on $\cW_0$,
and where we used the fact that the supersingular annuli acquire opposite orientations
with respect to $\cW_\infty$ and $\cW_0$. Combining $(\ref{00})$ and $(\ref{0})$,
we get
\begin{equation}\label{contrib-res}
0=\sum_{Q\in S}{\rm res}_{\cV_Q}\langle
F_{f,Q}^{[p]}\wedge\alpha,\eta_\Delta^{\rm frob}\rangle
=\left(1-\frac{a_p(f)}{p^{r+1}}\right)\sum_{Q\in S}{\rm res}_{\cV_Q}\langle F_{f,Q}\wedge\alpha,\eta_\Delta^{\rm frob}\rangle,
\end{equation}
where $F_{f,Q}^{[p]}$ denotes $F_{f,\infty}^{[p]}$ or $F_{f,0}^{[p]}$ depending on whether
$Q\in\cW_\infty$ or $\cW_0$, respectively, using $(\ref{frob})$ for the second equality
(see the argument \cite[p.1079]{bdp1}).

Since $a_p(f)^2=p^r$, this shows that there is no contribution from $\eta_\Delta^{\rm frob}$ in $(\ref{AJdelta})$.
On the other hand, since by the choice of $\eta_\Delta^{\rm hol}$ we easily have
\[
\sum_{Q\in S}{\rm res}_{\cV_Q}\langle
F_{f,Q}^{}\wedge\alpha,\eta_\Delta^{\rm hol}\rangle
=\langle F_{f,\infty}^{}(P_{A'})\wedge\alpha,{\rm cl}_{P_{A'}}(\Delta_\varphi)\rangle
\]
(see \cite[Lemma~3.19]{bdp1}), the result follows.
\end{proof}

Let $(A,t_A,\omega_A)$ be the CM triple introduced in Section~\ref{subsubsec:A},
and let $\eta_A\in H_{\rm dR}^1(A/F)$ be the class determined by the conditions
\[
\lambda^*\eta_A=\lambda^\rho\eta_A\quad\textrm{for all $\lambda\in\cO_K$,}\quad\quad\textrm{and}\quad\quad
\langle\omega_A,\eta_A\rangle_A=1,
\]
where $\lambda\mapsto\lambda^\rho$ denotes the action of the non-trivial automorphism of $K$,
and $\langle,\rangle_A$ is the cup product pairing on $H^1_{\rm dR}(A/F)$.
If 
$(A',t_{A'},\omega_{A'})$ is the CM triple induced
from $(A,t_{A},\omega_{A})$ by an isogeny $\varphi\in{\rm Isog}_c^{\mathfrak{N}}(A)$, we define
$\eta_{A'}\in H_{\rm dR}^1(A'/F)$ by the analogous recipe. For the integers $j$ with $0\leq j\leq r$,
the classes $\omega_{A'}^j\eta_{A'}^{r-j}$ defined in \cite[(1.4.6)]{bdp1} then form a basis of ${\rm Sym}^rH_{\rm dR}^1(A'/F)$.

\begin{lem}\label{lem:3.22}
Let the notations be as in Proposition~\ref{prop:AJ-Coleman}.
Then, for each $0\leq j\leq r$, we have
\[
{\rm AJ}_F(\Delta_\varphi)(\omega_f\wedge\omega_A^j\eta_A^{r-j})=
{\rm deg}(\varphi)^j\cdot\langle F_{f,\infty}(P_{A'}),\omega_{A'}^j\eta_{A'}^{r-j}\rangle_{A'},
\]
where $F_{f,\infty}$ is the Coleman primitive of $\omega_f\in H^0(X,\underline{\omega}^{\otimes r}\otimes\Omega^1_X)$
on $\cW_\infty$ (vanishing at $\infty$ if $r=0$),
and the pairing $\langle,\rangle_{A'}$ on the right-hand side is the natural one
on ${\rm Sym}^rH_{\rm dR}^1(A'/F)$.
\end{lem}

\begin{proof}
This follows from Proposition~\ref{prop:AJ-Coleman} as in \cite[Lemma~3.22]{bdp1}.
\end{proof}

Recall that if $f\in S_{k}(X)$ is a cusp form of weight $k$ and level $\Gamma$, then
$f\vert_{\cW_\infty}$ defines a $p$-adic modular form $f_p\in\mathcal{M}_k(N)$ of weight $k$ and tame level $N$.
Evaluated on a CM triple $(A',t_{A'},\omega_{A'})$ of conductor $c$ prime to $p$, we then have
\[
f_p(A',t_{A'},\omega_{A'})=f(A',t_{A'},\alpha_\pp',\omega_{A'}),
\]
where $\alpha_\pp':A'\longrightarrow A'/A'[\pp]$ is the $p$-isogeny
defined by the canonical subgroup of $A'$ (see Section~\ref{subsubsec:A}).
By abuse of notation, in the following we will denote $f_p$ also by $f$.
The map $V$ defined in $(\ref{V})$ yields an operator
$V:\mathcal{M}_k(N)\longrightarrow\mathcal{M}_k(N)$ on $p$-adic modular forms whose effect on
$q$-expansions is given by $q\mapsto q^p$. Let $a_p(f)$ be the $U_p$-eigenvalue of $f$,
and define the \emph{$p$-depletion} of $f$ by
\[
f^{[p]}:=f-a_p(f)Vf
\]
Letting $d=q\frac{d}{dq}:\mathcal{M}_k(N)\longrightarrow\mathcal{M}_{k+2}(N)$ be
the Atkin--Serre operator, for any integer $j$ the limit
\[
d^{-1-j}f^{[p]}:=\lim_{t\to -1-j}d^tf^{[p]}
\]
is a $p$-adic modular form of weight $k-2-j$ and tame level $N$ (see \cite[Thm.~5]{Serre350}).

\begin{lem}\label{3.24}
Let the notation be as in Proposition~\ref{prop:AJ-Coleman}.
Then for each $0\leq j\leq r$ there exists a locally analytic $p$-adic modular form $G_j$ of weight
$r-j$ and tame level $N$ such that
\begin{equation}\label{eq:G}
\langle F_{f,\infty}(P_{A'}),\omega_{A'}^j\eta_{A'}^{r-j}\rangle_{A'}=G_j(A',t_{A'},\omega_{A'}),
\end{equation}
where $F_{f,\infty}$ is the Coleman primitive of $\omega_f$ on $\cW_\infty$ (vanishing at $\infty$ if $r=0$), and
\begin{equation}\label{katz}
G_j(A',t_{A'},\omega_{A'})-\frac{a_p(f)}{p^{r-j+1}}G_j(\pp*(A',t_{A'},\omega_{A'}))=
j! d^{-1-j}f^{[p]}(A',t_{A'},\omega_{A'}).
\end{equation}
\end{lem}

\begin{proof}
The construction of $G_j$ as the ``$j$-th component'' of $F_{f,\infty}$ is given in \cite[p.1083]{bdp1},
and $(\ref{eq:G})$ then follows from the definition. On the the other hand, $(\ref{katz})$ follows
from the same calculations as in [\emph{loc.cit}, Lemma~3.23 and Prop.~3.24].
\end{proof}

We now relate the expression appearing in the right-hand side of
Proposition~\ref{prop:AJ-Coleman} to the value of a certain $p$-adic $L$-function associated to $f$.
\sk

Recall that $(A,t_A,\omega_A)$ denotes the CM triple introduced in Section~\ref{subsubsec:A}, and
fix an elliptic curve $A_0/H_c$ with ${\rm End}_{H_c}(A_c)\simeq\cO_c$.
The curve $A_0$ is related to $A$ by an isogeny $\varphi_0:A\longrightarrow A_0$ in
${\rm Isog}_c^{\mathfrak{N}}(A)$, and we let $(A_0,t_0,\omega_0)$
be the induced triple. 
Since we assume that $p=\pp\overline{\pp}$ splits in $K$, we may fix an isomorphism
$\mu_{p^\infty}\simeq\mathcal{A}_0[\pp^\infty]$ of $p$-divisible groups, where
$\mathcal{A}_0/\cO_{\bC_p}$ is a good integral model of $A_0$. This amounts to the
choice of an isomorphism $\imath:\hat{\mathcal{A}}_0\longrightarrow\hat{\mathbb{G}}_m$ of formal groups,
and we let $\Omega_p\in\bC_p^\times$ be the $p$-adic period defined by the rule
\[
\omega_0=\Omega_p\cdot\omega_{\rm can},
\]
where $\omega_{\rm can}:=\imath^*\frac{dt}{t}$ for the standard coordinate $t$ on $\hat{\mathbb{G}}_m$.
\sk

Finally, consider the set $\Sigma_{k,c}^{+}$ of algebraic Hecke characters
$\chi:K^\times\backslash\mathbb{A}_K^\times\longrightarrow\bC^\times$ of conductor $c$,
infinity type $(k+j,-j)$ with $j\geq 0$ (with the convention in \cite[p.1089]{bdp1}),
and such that
\[
\chi\vert_{\mathbb{A}_\bQ^\times}=\boldsymbol{\rm N}^{k},
\]
where $\boldsymbol{\rm N}$ is the norm character on $\mathbb{A}_\bQ^\times$, and for every
$\chi\in\Sigma_{k,c}^+$ 
set
\begin{equation}\label{def:Lp}
L_\pp(f)(\chi):=\sum_{[\fa]\in{\rm Pic}(\cO_c)}\chi^{-1}(\fa){\rm N}(\fa)^{-j}
\cdot d^jf^{[p]}(\fa*(A_0,t_0,\omega_{\rm can})),
\end{equation}
and define
\begin{equation}\label{eq:alg}
L_{\rm alg}(f,\chi^{-1}):=w(f,\chi)^{-1}C(f,\chi,c)\cdot\frac{L(f,\chi^{-1},0)}{\Omega^{2(k+2j)}},\nonumber
\end{equation}
where $w(f,\chi)$ and $C(f,\chi,c)$ are the constants defined in \cite[(5.1.11)]{bdp1}
and [\emph{loc.cit.}, Thm.~4.6], respectively, $\Omega$ is the complex period in [\emph{loc.cit.}, (5.1.16)],
and $L(f,\chi^{-1},0)$ is the central critical value of the Rankin--Selberg $L$-function
$L(f\times\theta_{\chi^{-1}},s)$ of $f$ and the theta series of $\chi^{-1}$.
\sk

As explained in \cite[p.1134]{bdp1}, the set $\Sigma_{k,c}^{+}$
may be endowed with a natural $p$-adic topology,
and we let $\hat{\Sigma}_{k,c}$ denote its completion.

\begin{thm}\label{5.27-5.28}
The assignment $\chi\mapsto L_\pp(f)(\chi)$ extends to a continuous function
on $\hat{\Sigma}_{k,c}$ and satisfies the following interpolation property.
If $\chi\in\Sigma_{k,c}^+$ has infinity type $(k+j,-j)$, with $j\geq 0$, then
\[
\frac{L_\pp(f)(\chi)^2}{\Omega_p^{2(k+2j)}}=(1-a_p(f)\chi^{-1}(\bar{\pp}))^2
\cdot L_{\rm alg}(f,\chi^{-1},0).
\]
\end{thm}

\begin{proof}
See Theorem~5.9, Proposition~5.10, and equation $(5.2.4)$ of \cite{bdp1},
noting that $\beta_p=0$ here, since $f$ has level divisible by $p$.
\end{proof}

Let $\Sigma_{k,c}^-$ be the set of
algebraic Hecke characters of $K$ of conductor $c$
and infinity type $(k-1-j,1+j)$, with $j\geq 0$.
Even though $\Sigma_{k,c}^+\cap\Sigma_{k,c}^-=\emptyset$,
any character in $\Sigma_{k,c}^-$ can be written as a limit of characters in $\Sigma_{k,c}^+$ (see \cite[p.1137]{bdp1}).
Thus for any $\chi\in\Sigma_{k,c}^-$, the value $L_\pp(f)(\chi)$ is defined by continuity.
\sk

The next result extends the $p$-adic Gross--Zagier formula of \cite[Thm.~5.13]{bdp1}
to the semistable non-crystalline setting.

\begin{thm}\label{thmbdp1A}
Let $f=q+\sum_{n=2}^{\infty}a_n(f)q^n\in S_{k}(\Gamma_0(Np))$ be a $p$-new eigenform
of weight $k=r+2\geq 2$, 
and suppose that $\chi\in\Sigma_{k,c}^-$ has infinity type $(r+1-j,1+j)$,
with $0\leq j\leq r$. Then
\begin{equation}\label{formula:bdp1A}
\frac{L_\pp(f)(\chi)}{\Omega_p^{r-2j}}=
(1-a_p(f)\chi^{-1}(\bar{\pp}))\cdot
\biggl(\frac{c^{-j}}{j!}\sum_{[\fa]\in{\rm
Pic}(\cO_c)}\chi^{-1}(\fa){\rm N}(\fa)\cdot
{\rm AJ}_F(\Delta_{\varphi_\fa\varphi_0})
(\omega_{f}\wedge\omega_A^{j}\eta_A^{r-j})\biggr).\nonumber
\end{equation}
\end{thm}

\begin{proof}
The proof of \cite[Prop.~5.10]{bdp1} shows that the expression $(\ref{def:Lp})$
extends in the obvious way to a character
$\chi$ as in the statement, yielding
\begin{equation}\label{takelim}
\frac{L_\pp(f)(\chi)}{\Omega_p^{r-2j}}=
\sum_{[\mathfrak{a}]\in{\rm
Pic}(\cO_c)}\chi^{-1}(\fa){\rm N}(\fa)^{1+j}
\cdot d^{-1-j}f^{[p]}(\mathfrak{a}*(A_0,t_0,\omega_0)).
\end{equation}
On the other hand, by Lemma~\ref{3.24} we have
\begin{equation}\label{from3.24}
j! d^{-1-j}f^{[p]}(\mathfrak{a}*(A_0,t_0,\omega_0))=G_j(\fa*(A_0,t_0,\omega_0))-\frac{a_p(f)}{p^{r-j+1}}G_j(\pp\fa*(A_0,t_0,\omega_0)).
\end{equation}
Substituting $(\ref{from3.24})$ into $(\ref{takelim})$, summing over $[\fa]\in{\rm Pic}(\cO_c)$,
and noting that
\[
\chi(\pp)p^{-1-j}=\chi^{-1}(\overline\pp)p^{r+1-j},
\]
we see that
\begin{equation}\label{G}
\frac{L_\pp(f)(\chi)}{\Omega_p^{r-2j}}=\left(1-a_p(f)\chi^{-1}(\overline\pp)\right)
\cdot\biggl(\frac{1}{j!}
\sum_{[\fa]\in{\rm Pic}(\cO_c)}\chi^{-1}(\fa){\rm N}(\fa)^{1+j}\cdot G_j(\mathfrak{a}*(A_0,t_0,\omega_0))\biggr).
\end{equation}
Finally, since the isogeny
$\varphi_\fa\varphi_0:(A,t_A,\omega_A)\longrightarrow\fa*(A_0,t_0,\omega_0)$
has degree $c{\rm N}(\fa)$, combining Lemma~\ref{lem:3.22} and Lemma~\ref{3.24} we have
\begin{equation}\label{modifAJ}
G_j(\fa*(A_0,t_0,\omega_0))=c^{-j}{\rm N}(\fa)^{-j}\cdot
{\rm AJ}_F(\Delta_{\varphi_\fa\varphi_0})(\omega_f\wedge\omega_A^j\eta_A^{r-j}),
\end{equation}
and substituting $(\ref{modifAJ})$ into $(\ref{G})$, the result follows.
\end{proof}

\section{Main result}

In this section we prove the main result of this paper, 
giving an ``exceptional zero formula'' for the specializations
of Howard's big Heegner points at exceptional primes in the Hida family.

\subsection{Heegner points in Hida families}
\label{subsec:bigHP}

We begin by briefly reviewing the constructions of \cite{howard-invmath},
which we adapt to our situation, referring the reader to [\emph{loc.cit.}, \S{2}] for further details.
\sk

Recall that $p\nmid N$. Let $f=\sum_{n=1}^\infty a_n(f)q^n\in S_k(\Gamma_0(Np))$
be a newform, fix a finite extension $L$ of $\bQ_p$ with ring of integers $\cO_L$
containing the Fourier coefficients of $f$, and let
\[
\rho_f:G_\bQ:={\rm Gal}(\overline{\bQ}/\bQ)\longrightarrow{\rm Aut}_L(V_f)\simeq\mathbf{GL}_2(L)
\]
be the Galois representation associated to $f$. Also, let $K=\bQ(\sqrt{-D_K})$
be an imaginary quadratic field as in $\S\ref{subsubsec:A}$.
For the rest of this paper, these will be subject to the following further hypotheses.

\begin{ass}\label{running-ass}
\begin{enumerate}
\item{} $f$ is ordinary at $p$, i.e., $\imath_p(a_p(f))$ is a $p$-adic unit;
\item{} $\overline{\rho}_f$ is absolutely irreducible;
\item{} $\overline{\rho}_f$ is ramified at every prime $q$ dividing $(D_K,N)$;
\item{} $p\nmid h_K:=\vert{\rm Pic}(\cO_K)\vert$, the class number of $K$.
\end{enumerate}
\end{ass}

Note that by \cite[Lemma~2.15]{howard-invmath}, the first assumption forces the weight of $f$
to be $k=2$, which will thus be assumed for the rest of this paper. 

\begin{defn}
Set $\Lambda_{\cO_L}:=\cO_L\pwseries{1+p\bZ_p}$.
For any $\Lambda_{\cO_L}$-algebra $A$, let $\mathcal{X}_{\cO_L}^a(A)$ be
the set of continuous $\cO_L$-algebra homomorphisms
$\nu:A\longrightarrow\overline{\bQ}_p$ such that the composition
\[
1+p\bZ_p\longrightarrow\Lambda_{\cO_L}^\times\longrightarrow A^\times
\xrightarrow{\;\;\nu\;\;}\overline{\bQ}_p^\times
\]
is given by $\gamma\mapsto\gamma^{k_\nu-2}$, for some integer $k_\nu\geq 2$ with $k_\nu\equiv 2\pmod{2(p-1)}$
called the \emph{weight} of $\nu$.
\end{defn}

Since $f$ is ordinary at $p$, by \cite[Cor.~1.3]{hida86b})
there exists a local reduced finite integral extension
$\cR$ of $\Lambda_{\cO_L}$, and a formal $q$-expansion
$\sF=\sum_{n=1}^\infty\mathbf{a}_nq^n\in\cR[[q]]$ uniquely characterized by the following property.
For every $\nu\in\mathcal{X}^a_{\cO_L}(\cR)$ of weight $k_\nu>2$, there exists a newform
$f_\nu\in S_{k_\nu}(\Gamma_0(N))$ such that
\begin{equation}\label{p-stab}
\nu(\sF)=f_\nu(q)-\frac{p^{k_\nu-1}}{\nu(\mathbf{a}_p)}f_\nu(q^p),
\end{equation}
and we there exists a unique $\nu_f\in\mathcal{X}^a_{\cO_L}(\cR)$
of weight $2$ such that ${\nu_f}(\sF)=f(q)$.
\sk


By \cite[Thm.~1.2]{hida86b}, there is
a free $\cR$-module $\mathbf{T}$ of rank $2$ equipped with
a continuous action
\[
\rho_\sF:G_\bQ\longrightarrow
{\rm Aut}_\cR(\mathbf{T})\cong\mathbf{GL}_2(\cR)
\]
such that for every $\nu\in\mathcal{X}_{\cO_L}^a$,
$\nu(\rho_\sF)$ is isomorphic to the Galois representation
$\rho_{f_\nu}:G_\bQ\longrightarrow\mathbf{GL}_2(\overline{\bQ}_p)$
associated to $f_\nu$. Moreover, by \cite[Thm.~2.2.2]{wiles88},
if $D_p\subset G_{\bQ}$ is the decomposition group of any place $\mathfrak{P}$
of $\overline{\bQ}$ above $p$, 
there exists an exact sequence of $\cR[D_p]$-modules
\begin{equation}\label{eq:ord}
0\longrightarrow\fil^+\mathbf{T}\longrightarrow\mathbf{T}\longrightarrow\fil^-\mathbf{T}\longrightarrow 0
\end{equation}
with $\fil^{\pm}\mathbf{T}$ free of rank $1$ over $\cR$, and with the $D_p$-action on $\fil^-\mathbf{T}$
given by the unramified character sending an arithmetic Frobenius ${\rm Fr}_p^{-1}$ to $\mathbf{a}_p\in\cR^\times$.
\sk

Following \cite[Def.~2.1.3]{howard-invmath}, define the \emph{critical character}
$\Theta:G_\bQ\longrightarrow\cR^\times$ by the composite
\begin{equation}\label{def:crit}
\Theta:G_\bQ\xrightarrow{\varepsilon_{\rm cyc}}\bZ_p^\times\xrightarrow{\;\langle\cdot\rangle\;}
1+p\bZ_p\xrightarrow{\gamma\mapsto\gamma^{1/2}}1+p\bZ_p
\longrightarrow\Lambda_\cO^\times\longrightarrow\cR^\times,
\end{equation}
where $\varepsilon_{\rm cyc}$ is the $p$-adic cyclotomic character,
$\langle\cdot\rangle$ denotes the projection to the $1$-units in $\bZ_p$.
Let $\cR^\dagger$ be the free $\cR$-module of rank $1$
where $G_\bQ$ acts via $\Theta^{-1}$, and set
\[
\mathbf{T}^\dagger:=\mathbf{T}\otimes_{\cR}\cR^\dagger
\]
equipped with the diagonal Galois action.
Then, if for every $\nu\in\mathcal{X}_{\cO_L}^{a}(\cR)$
we let $V_{f_\nu}$ be a representation space for $\rho_{f_\nu}$,
then $\nu(\mathbf{T}^\dagger):=\mathbf{T}^\dagger\otimes_{\cR,\nu}\nu(\cR)$ is
isomorphic to a lattice in the self-dual Tate twist
$V_{f_\nu}(k_\nu/2)$ of $V_{f_\nu}$ (see 
\cite[Thm.~1.4.3]{Ohta} and \cite[(3.2.4)]{Nekovar-Plater}). 
\sk

Let $K_\infty$ be the anticyclotomic $\bZ_p$-extension of $K$, and for each $n\geq 0$,
let $K_n$ be the subfield of $K_\infty$ with ${\rm Gal}(K_n/K)\simeq\bZ/p^n\bZ$.

\begin{thm}[Howard]
\label{thm:bigHPs}
There is a system of ``big Heegner points''
\[
\mathfrak{Z}_\infty=\{\mathfrak{Z}_n\}_{n\geq 0}\in
H^1_{\rm Iw}(K_\infty,\mathbf{T}^\dagger)
:=\varprojlim_nH^1(K_n,\mathbf{T}^\dagger)
\]
with the following properties.
\begin{enumerate}
\item{} For each $n$, $\mathfrak{Z}_n$ belongs to the Greenberg Selmer group
${\rm Sel}_{\rm Gr}(K_n,\mathbf{T}^\dagger)$ of \cite[Def.~2.4.2]{howard-invmath}.
In particular, for every prime $\mathfrak{q}$ of $K$ above $p$, we have
\[
{\rm loc}_{\mathfrak{q}}(\mathfrak{Z}_\infty)\in{\rm ker}\left(H^1_{\rm Iw}(K_{\infty,\mathfrak{q}},\mathbf{T}^\dagger)
\longrightarrow H^1_{\rm Iw}(K_{\infty,\mathfrak{q}},\fil^-\mathbf{T}^\dagger)\right)
\]
for the natural map induced by $(\ref{eq:ord})$.
\item{} If $\mathfrak{Z}_\infty^*$ denotes the image of $\mathfrak{Z}_\infty$ under the action of complex
conjugation, then
\begin{equation}\label{eq:w}
\mathfrak{Z}_\infty^*=w\cdot\mathfrak{Z}_\infty\nonumber
\end{equation}
for some $w\in\{\pm{1}\}$.
\end{enumerate}
\end{thm}

\begin{proof}
In the following, all the references are to \cite{howard-invmath}.
The construction of $\mathfrak{Z}_\infty$ is given in
\S\S{2.2}, 3.3 and the proof of $(1)$ is given in Prop.~2.4.5.
For the proof of $(2)$, we need to briefly recall the definition of $\mathfrak{Z}_n$.
Let $H_{p^{n+1}}$ be the ring class field of $K$ of conduction $p^{n+1}$, and note that it contains $K_n$.
By Prop.~2.3.1, the ``big Heegner points'' $\mathfrak{X}_{p^{n+1}}\in H^1(H_{p^{n+1}},\mathbf{T}^\dagger)$
safisfy ${\rm Cor}_{H_{p^{n+1}}/H_{p^n}}(\mathfrak{X}_{p^{n+1}})=U_p\cdot\mathfrak{X}_{p^n}$,
and hence the classes
\begin{equation}\label{def:Z}
\mathfrak{Z}_n:=U_p^{-n}\cdot{\rm Cor}_{H_{p^{n+1}}/K_n}(\mathfrak{X}_{p^{n+1}})
\end{equation}
are compatible under corestriction. Denoting by $\tau$ the image of a class under
the action of complex conjugation and using Prop.~2.3.5, we find that
\begin{align}\label{eq:dih}
{\rm Cor}_{H_{p^{n+1}}/K_n}(\mathfrak{X}_{p^{n+1}})^\tau
&=\sum_{\sigma\in{\rm Gal}(H_{p^{n+1}}/K_n)}\mathfrak{X}_{p^{n+1}}^{\tau\sigma}\\
&=\sum_{\sigma\in{\rm Gal}(H_{p^{n+1}}/K_n)}\mathfrak{X}_{p^{n+1}}^{\sigma^{-1}\tau}\nonumber\\
&=w\cdot{\rm Cor}_{H_{p^{n+1}}/K_n}(\mathfrak{X}_{p^{n+1}})\nonumber
\end{align}
for some $w\in\{\pm 1\}$. Combining $(\ref{def:Z})$ and $(\ref{eq:dih})$, the result follows.
\end{proof}

\subsection{Two-variable $p$-adic $L$-functions}
\label{sec:2varL}


As in the preceding section, let $f\in S_2(\Gamma_0(Np))$ be a newform split multiplicative at $p$,
and let $\mathbf{f}\in\cR\pwseries{q}$ be the Hida family passing through $f$.
Recall the spaces of characters $\Sigma_{k,c}^{\pm}$ and $\hat{\Sigma}_{k,c}$ introduced in Section~\ref{sec:computations}.
In the following, we only consider the case $c=1$,
which will henceforth be suppressed from the notation.
\sk

By \cite[Prop.~5.10]{bdp1} (see also Theorem~\ref{5.27-5.28}),
for every $\nu\in\mathcal{X}_{\cO_L}^a(\cR)$ 
the assignment
\[
\chi\longmapsto L_\pp(f_\nu)(\chi):=\sum_{[\fa]\in{\rm Pic}(\cO_K)}\chi^{-1}(\fa){\rm N}(\fa)^{-j}
\cdot d^jf_\nu^{[p]}(\fa*(A_0,t_0,\omega_{\rm can}))
\]
extends to a continuous function on $\hat{\Sigma}_{k_\nu}$.
Using the explicit expression for these values,
it is easy to show the existence of a two-variable $p$-adic $L$-function interpolating
$L_\pp(f_\nu)$ for varying $\nu$. For the precise statement, denote by $h=h_K$ the class number of $K$
(which we assume is prime to $p$), and let $\phi_o$ 
be the unramified Hecke character  
defined on fractional ideals by the rule
\begin{equation}\label{phi}
\phi_o(\mathfrak{a})=\alpha/\overline{\alpha},\quad\textrm{where $(\alpha)=\mathfrak{a}^h$}.
\end{equation}
Assume that $\cO_L$ contains the values of $\phi_o$,
and denote by $\langle\phi_o\rangle$ the composition of $\phi_o$
with the projection onto the $\bZ_p$-free quotient of $\cO_L^\times$,
which then is valued in $1+p\bZ_p$, and define $\xi:K^\times\backslash\mathbb{A}_K^\times\longrightarrow\cR^\times$ by
\begin{equation}\label{def:xi}
\xi:K^\times\backslash\mathbb{A}_K^\times\xrightarrow{\;\;\phi_o\;\;}\cO_L^\times\xrightarrow{\;\langle\cdot\rangle\;}
1+p\bZ_p\xrightarrow{\gamma\mapsto\gamma^{1/2h}}1+p\bZ_p
\longrightarrow\Lambda_{\cO_L}^\times\longrightarrow\cR^\times.
\end{equation}
Similarly, recall the critical character $\Theta:G_{\bQ}\longrightarrow\cR^\times$ from $(\ref{def:crit})$,
and define $\chi:K^\times\backslash\mathbb{A}_K^\times\longrightarrow\cR^\times$ by
\[
\chi(x)=\Theta({\rm rec}_\bQ({\rm N}_{K/\bQ}(x))),
\]
where ${\rm rec}_{\bQ}:\mathbb{A}_\bQ^\times\longrightarrow{\rm Gal}(\bQ^{\rm ab}/\bQ)$ is
the \emph{geometrically} normalized global reciprocity map.
Let $\Gamma_\infty:={\rm Gal}(K_\infty/K)$ be the Galois
group of the anticyclotomic $\bZ_p$-extension of $K$,
and denote by $\mathcal{X}_{\cO_L}^a(\Gamma_\infty)$
the set of continuous $\cO_L$-algebra homomorphisms $\cO_L\pwseries{\Gamma_\infty}\longrightarrow\bQ_p^\times$
induced by a character $\phi$ of the form $\phi=\phi_o^{\ell_\phi}$ for some integer
$\ell_\phi\geq 0$ with $\ell_\phi\equiv 0\pmod{p-1}$ called the \emph{weight} of $\phi$.
Finally, let
\[
{\rm\mathbf{N}}_K:K^\times\backslash\mathbb{A}_K^\times\xrightarrow{{\rm N}_{K/\bQ}}\bQ^\times\backslash\mathbb{A}_\bQ^\times
\xrightarrow{\;\;\rm\mathbf{N}\;\;}\bC^\times
\]
be the norm character of $K$, and for every $\nu\in\mathcal{X}_{\cO_L}^a(\cR)$,
let $\xi_\nu$ and $\chi_\nu$ 
be the composition of $\xi$ and $\chi$ with $\nu$, respectively.

\begin{thm}\label{prop:bigL}
The exists a continuous function $L_{\pp,\xi}(\mathbf{f})$
on $\mathcal{X}_{\cO_L}^a(\cR)\times\mathcal{X}_{\cO_L}^a(\Gamma_\infty)$
such that for every $\nu\in\mathcal{X}_{\cO_L}^a(\cR)$ we have
\[
L_{\pp,\xi}(\sF)(\nu,\phi)=L_\pp(f_\nu)(\phi\xi_\nu\chi_\nu{\rm\mathbf{N}}_K)
\]
as functions of $\phi\in\mathcal{X}^a_{\cO_L}(\Gamma_\infty)$.
\end{thm}

\begin{proof}
See \cite[Thm.~1.4]{cas-2var}. (Note that if
$(\nu,\phi)\in\mathcal{X}_{\cO_L}^a(\cR)\times\mathcal{X}_{\cO_L}^a(\Gamma_\infty)$,
then $\phi\xi_\nu\chi_\nu{\rm\mathbf{N}}_K$ is an unramified Hecke character of infinity type
$(k_\nu+\ell_\phi-1,1-\ell_\phi)$, thus lying in the domain of $L_\pp(f_\nu)$.)
\end{proof}

\subsection{Big logarithm maps}

By our assumption that $p\nmid h_K$, the extension $K_\infty/K$ is
totally ramified at every prime $\qq$ above $p$; let $K_{\infty,\qq}$ be the completion of $K_\infty$ at
the unique prime above $\qq$, and set $\Gamma_{\qq,\infty}:={\rm Gal}(K_{\infty,\qq}/K_\qq)$.
Even though $\Gamma_{\qq,\infty}$ may be identified with $\Gamma_\infty$, in the following it will
be convenient to maintain the distinction between them.
\sk

Recall the $\cR$-adic Hecke character introduced in $(\ref{def:xi})$, and 
let $\xi:G_K\longrightarrow\cR^\times$ also denote the Galois character defined by
\[
\xi(\sigma):=[\langle\hat{\phi}_o(\sigma)\rangle^{1/2h}],
\]
where $\hat\phi_o:G_K\longrightarrow\cO_L^\times$ is the $p$-adic avatar of the Hecke character
$\phi_o$ in $(\ref{phi})$. Finally, set
\[
\mathbb{T}:=\mathbf{T}^\dagger\vert_{G_K}\otimes\xi^{-1}, 
\]
and for every $\nu\in\mathcal{X}_{\cO_L}^a(\cR)$ denote by $V_\nu$
the specialization of $\mathbb{T}$ at $\nu$.

\begin{thm}\label{cor:L-}
Let $\qq\in\{\pp,\overline{\pp}\}$, define
$\lambda_\pm:=\mathbf{a}_p\cdot\Theta\xi^{\pm}({\rm Fr}_\qq)-1\in\cR$ and
set $\tilde{\cR}=\cR[\lambda_+^{-1}\lambda_-^{-1}]\otimes_{\bZ_p}\hat{\bZ}_p^{\rm nr}$. 
There exists an $\cR\pwseries{\Gamma_{\qq,\infty}}$-linear map
\[
\mathcal{L}_{\fil^+\mathbb{T}}^{\omega}:H^1_{\rm Iw}(K_{\infty,\qq},\fil^+\bT)
\longrightarrow\tilde{\cR}\pwseries{\Gamma_{\qq,\infty}}
\]
such for every $\mathfrak{Y}_\infty\in H^1_{\rm Iw}(K_{\infty,\mathfrak{q}},\fil^+\bT)$ and every
$(\nu,\phi)\in\mathcal{X}^a_{\cO_L}(\tilde{\cR})\times\mathcal{X}_{\cO_L}^a(\Gamma_{\infty})$, we have
\begin{align*}
\left(1-\frac{\Theta_\nu^{-1}\xi_\nu\phi^{-1}({\rm Fr}_{\mathfrak{q}})}{\nu(\mathbf{a}_p)}\right)
\mathcal{L}_{\fil^+\mathbb{T}}^{\omega}(\mathfrak{Y}_\infty)(\nu,\phi_\qq)
&=\ell_\phi!^{-1}\left(1-\frac{\nu(\mathbf{a}_p)p^{-1}}{\Theta_\nu^{-1}\xi_\nu^{}\phi^{-1}({\rm Fr}_{\qq})}\right)
\langle{\rm log}_{}(\nu(\mathfrak{Y}_\infty)^{\phi^{}}),\breve{\omega}_{\nu}\rangle_{},
\end{align*}
where $\log=\log_{\fil^+V_\nu\otimes\phi}:H^1(K_\qq,\fil^+V_\nu\otimes\phi)\longrightarrow D_{\rm dR}(\fil^+V_\nu\otimes\phi)$ 
is the Bloch--Kato logarithm map, and $\nu(\mathfrak{Y}_\infty)^{\phi^{}}\in H^1(K_{\qq},\fil^+V_\nu\otimes\phi^{})$
is the $\phi^{}$-specialization of $\nu(\mathfrak{Y}_\infty)$.
\end{thm}

\begin{proof}
See \cite[Prop.~4.3]{cas-2var}.
\end{proof}

\begin{rem}\label{D}
Fix a compatible system $\zeta_\infty=\{\zeta_r\}_{r\geq 0}$
of $p$-power roots of unity, and let $\zeta_\infty t^{-1}$ be the associated basis
element of $D_{\rm dR}(\bQ_p(1))$. In Theorem~\ref{cor:L-} above, 
$\omega$ denotes a generator of the module
\[
\mathbb{D}:=(\fil^+\bT(-1)\hat{\otimes}_{\bZ_p}\hat{\bZ}_p^{\rm nr})^{G_{K_\qq}},
\]
which by \cite[Lemma~3.3]{Ochiai-Col} is free of rank one over $\cR$. 
(Note that $\fil^+\bT(-1)$ is unramified.) 
As explained in \emph{loc.cit.}, for each $\nu\in\mathcal{X}_{\cO_L}^a(\cR)$ 
there is a specialization map
\[
\nu_*:\mathbb{D}\longrightarrow\mathbb{D}_\nu\otimes_{\bZ_p}\bQ_p\simeq D_{\rm dR}(\fil^+V_\nu(-1)).
\]
Then, letting $\omega_\nu$ denote the image of 
$\nu_*(\omega)\otimes\zeta_\infty t^{-1}$ in 
$D_{\rm dR}(\fil^+V_\nu(-1))\otimes D_{\rm dR}(\bQ_p(1))\simeq D_{\rm dR}(\fil^+V_\nu)$, 
the class $\breve{\omega}_\nu\in D_{\rm dR}(\fil^-V_\nu^*(1))$ in the above interpolation formulae is defined 
by requiring that
\[
\langle\omega_\nu,\breve{\omega}_\nu\rangle=1
\]
under the de Rham pairing $\langle,\rangle:D_{\rm dR}(\fil^+V_\nu)\times D_{\rm dR}(\fil^-V_\nu^*(1))\longrightarrow F_\nu$.
\end{rem}

The big logarithm map $\mathcal{L}_{\fil^+\bT}^\omega$ of Theorem~\ref{cor:L-} may not be specialized
at any pair $(\nu,\mathds{1})$ with $\nu\in\mathcal{X}_{\cO_L}^a(\cR)$ such that $\nu(\lambda_\pm)=0$. 
Since such arithmetic primes are in fact the main concern
in this paper, the following construction of an ``improved'' big logarithm map 
will be useful.

\begin{prop}\label{prop:improved}
There exists an $\cR$-linear map
\[
\bar{\mathcal{L}}^\omega_{\fil^+\bT}:H^1(K_\pp,\fil^+\bT)\longrightarrow\cR\otimes_{\bZ_p}\bQ_p
\]
such that for every $\mathfrak{Y}_0\in H^1(K_\pp,\fil^+\bT)$ and
every $\nu\in\mathcal{X}^a_{\cO_L}(\cR)$, we have
\begin{align*}
\nu\left(\bar{\mathcal{L}}_{\fil^+\mathbb{T}}^{\omega}(\mathfrak{Y}_0)\right)
&=\left(1-\frac{\nu(\mathbf{a}_p)p^{-1}}{\Theta^{-1}\xi^{}({\rm Fr}_\pp)}\right)
\langle{\rm log}_{\fil^+V_\nu}(\nu(\mathfrak{Y}_0)),\breve{\omega}_{\nu}\rangle_{}.
\end{align*}
\end{prop}

\begin{proof}
This can be shown by adapting the methods of Ochiai~\cite[\S{5}]{Ochiai-Col}.
Indeed, let
\[
\mathcal{L}_{\fil^+\bT}:H^1(K_\pp,\fil^+\bT)\otimes\bQ_p\longrightarrow\mathbb{D}\otimes_{\bZ_p}\bQ_p
\]
be the inverse of the map ${\rm exp}_{\bT}$ constructed in \cite[Prop.~3.8]{venerucci-exp} 
(see Remark~\ref{D} for the definition of $\mathbb{D}$), 
and define
\[
\mathcal{L}_{\fil^+\bT}^\omega:H^1(K_\pp,\fil^+\bT)\longrightarrow\cR\otimes_{\bZ_p}\bQ_p
\]
by the relation $\mathcal{L}_{\fil^+\bT}(-)=\mathcal{L}_{\fil^+\bT}^{\omega}(-)\cdot\omega$.
Setting
\[
\bar{\mathcal{L}}_{\fil^+\bT}^\omega=\biggl(1-\frac{\mathbf{a}_pp^{-1}}{\Theta^{-1}\xi^{}({\rm Fr}_\pp)}\biggr)
\mathcal{L}_{\fil^+\bT}^\omega:
H^1(K_\pp,\fil^+\bT)\longrightarrow\cR\otimes_{\bZ_p}\bQ_p,
\]
the result follows.
\end{proof}

\begin{cor}\label{cor:factor}
For any $\mathfrak{Y}_\infty=\{\mathfrak{Y}_n\}_{n\geq 0}\in H^1_{\rm Iw}(K_{\infty,\pp},\fil^+\bT)$
we have the factorization in $\tilde{\cR}$:
\[
\left(1-\frac{\Theta^{-1}\xi^{}({\rm Fr}_\pp)}{\mathbf{a}_p}\right)
\cdot\varepsilon\left(\mathcal{L}_{\fil^+\bT}^\omega(\mathfrak{Y}_\infty)\right)=
\bar{\mathcal{L}}_{\fil^+\bT}^\omega(\mathfrak{Y}_0),
\]
where $\varepsilon:\tilde{\cR}\pwseries{\Gamma_\infty}\longrightarrow\tilde{\cR}$ is the augmentation map.
\end{cor}

\begin{proof}
Comparing the interpolation formulas in Theorem~\ref{cor:L-} and Proposition~\ref{prop:improved}, we
see that
\[
\left(1-\frac{\Theta_\nu^{-1}\xi_\nu^{}({\rm Fr}_\pp)}{\nu(\mathbf{a}_p)}\right)
\mathcal{L}_{\fil^+\bT}^\omega(\mathfrak{Y}_\infty)(\nu,\mathds{1})=
\nu\left(\bar{\mathcal{L}}_{\fil^+\bT}^\omega(\mathfrak{Y}_0)\right)
\]
for every $\nu\in\mathcal{X}_{\cO_L}^a(\tilde{\cR})$;
since these primes are dense in $\tilde{\cR}$, 
the corollary follows.
\end{proof}

The proof of our main result will rely crucially on the relation found in \cite[\S{4}]{cas-2var}
between the $p$-adic $L$-function $L_{\pp,\xi}(\sF)$ of Theorem~\ref{prop:bigL}
and Howard's system of big Heegner points $\mathfrak{Z}_\infty$. 
We conclude this section by briefly recalling that relation.
\sk

By \cite[Lemma~2.4.4]{howard-invmath}, for every prime $\qq$ of $K$ above $p$
the natural map
\[
H_{\rm Iw}^1(K_{\infty,\qq},\fil^+\mathbf{T}^\dagger)\longrightarrow H_{\rm Iw}^1(K_{\infty,\qq},\mathbf{T}^\dagger)
\]
induced by $(\ref{eq:ord})$ is injective. In light of Theorem~\ref{thm:bigHPs}, in the following
we will thus view ${\rm loc}_\qq(\mathfrak{Z}_\infty)$ as sitting inside $H^1_{\rm Iw}(K_{\infty,\qq},\fil^+\mathbf{T}^\dagger)$.

\begin{thm}\label{thm:equality}
There is a generator $\omega=\omega_{\mathbf{f}}$ of the module $\mathbb{D}$ such that
\[
\mathcal{L}^{\omega}_{\fil^+\bT}({\rm loc}_\pp(\mathfrak{Z}_\infty^{\xi^{-1}}))
=L_{\pp,\xi}(\F)
\]
as 
functions on $\mathcal{X}_{\cO_L}^a(\tilde{\cR})\times\mathcal{X}_{\cO_L}^a(\Gamma_\infty)$.
\end{thm}

\begin{proof}
The construction of the basis element $\omega=\omega_{\mathbf{f}}$ of $\mathbb{D}$ 
is deduced in \cite[Prop.~10.1.2]{KLZ2} from Ohta's work \cite{OhtaII}, 
and it has the property that $\langle\omega_\nu,\omega_{\mathbf{f}_\nu}\rangle=1$, 
for all $\nu\in\mathcal{X}_{\cO_L}^a(\cR)$, where $\omega_{\mathbf{f}_\nu}$ 
is the class in ${\rm Fil}^1D_{\rm dR}(V_\nu^*)\simeq D_{\rm dR}(\fil^-V_\nu^*(1))$ 
associated to the $p$-stabilized newform $(\ref{p-stab})$; in particular, 
\[
\breve{\omega}_{\nu_f}=\omega_f
\]
in the notations of Remark~\ref{D}. 
The result is then the content of \cite[Thm.~4.4]{cas-2var}.
\end{proof}

\subsection{Exceptional zero formula}

Let $f=\sum_{n=1}^\infty a_n(f)q^n\in S_2(\Gamma_0(Np))$ be an ordinary newform as in Section~\ref{subsec:bigHP},
and assume in addition that $f$ is \emph{split multiplicative} at $p$, meaning that $a_p(f)=1$.
Recall the CM triple $(A,t_A,\alpha_\pp)\in X(H)$
introduced in Section~\ref{subsubsec:A}, which maps to the point $P_A=(A,A[\mathfrak{Np}])\in X_0(Np)$
under the forgetful map $X\longrightarrow X_0(Np)$. Let $\infty$ be any cusp of $X_0(Np)$ rational over $\bQ$,
and let $\kappa_f\in H^1(K,V_f)$ be the image of $(P_A)-(\infty)$ under the composite map
\begin{equation}\label{eq:kummer}
J_0(Np)(H)\xrightarrow{{\rm Kum}}H^1(H,{\rm Ta}_p(J_0(Np))\otimes_{\bZ_p}\bQ_p)
\longrightarrow H^1(H,V_f)\xrightarrow{{\rm Cor}_{H/K}}H^1(K,V_f).
\end{equation}

If $\sF\in\cR[[q]]$ is the Hida family passing through $f$, and $\nu_f\in\mathcal{X}^a_{\cO_L}(\cR)$
is the arithmetic prime of $\cR$ such that $\nu_f(\sF)=f$, it would be natural to expect a relation between
the class $\kappa_f$ and the specialization at $\nu_f$ of Howard's big Heegner point
$\mathfrak{Z}_0$. 
As done in \cite[\S{3}]{howard-mathann}, one can 
trace through the construction of $\mathfrak{Z}_0$ to deduce a relation between the \emph{generic} 
(in the sense of [\emph{loc.cit.}, Def.~2]) weight $2$ specializations of $\mathfrak{Z}_0$ and the Kummer images of certain CM points.
However, 
the arithmetic prime $\nu_f$ is not generic in that sense, 
and in fact one does not expect a similar direct relation between
$\nu_f(\mathfrak{Z}_0)$ and $\kappa_f$ (see the discussion in [\emph{loc.cit.}, p.813]).
\sk

In Theorem~\ref{main} below we will show that in fact
the localization at $\pp$ of $\nu_f(\mathfrak{Z}_0)$ vanishes,
but that nonetheless it can be related to $\kappa_f$
upon taking a certain ``derivative'' in the following sense, where we let
$\log_p:\bQ_p^\times\longrightarrow\bQ_p$ be Iwasawa's branch of the $p$-adic logarithm.

\begin{lem}\label{lem:divide}
Let $T$ be a free $\cO_L$-module of finite rank equipped with a linear action of $G_{\bQ_p}$,
let $k_\infty/\bQ_p$ be a $\bZ_p$-extension, and let $\gamma\in{\rm Gal}(k_\infty/\bQ_p)$
be a topological generator. Assume that $T^{G_{k_\infty}}=\{0\}$, and let
$\mathcal{Z}_\infty=\{\mathcal{Z}_n\}_{n\geq 0}\in H_{\rm Iw}^1(k_\infty,T)$
be such that $\mathcal{Z}_0=0$.
Then there exists a unique $\mathcal{Z}_{\gamma,\infty}'
=\{\mathcal{Z}_{\gamma,n}'\}_{n\geq 0}\in H_{\rm Iw}^1(k_\infty,T)$
such that
\[
\mathcal{Z}_\infty=(\gamma-1)\cdot\mathcal{Z}_{\gamma,\infty}'.
\]
Moreover, if $\eta:{\rm Gal}(k_\infty/\bQ_p)\simeq\bZ_p$ is any group isomorphism,
then
\begin{equation}
\mathcal{Z}_0':=\frac{\mathcal{Z}'_{\gamma,0}}{\log_p(\eta(\gamma))}\in H^1(\bQ_p,T[1/p])\nonumber
\end{equation}
is independent of the choice of $\gamma$.
\end{lem}

\begin{proof}
Consider the module $T_\infty:=T\hat{\otimes}_{\cO_L}\cO_L\pwseries{{\rm Gal}(k_\infty/\bQ_p)}$ equipped
with the diagonal Galois action, where $G_{\bQ_p}$ acts on the second factor via the projection
$G_{\bQ_p}\longrightarrow{\rm Gal}(k_\infty/\bQ_p)$. By Shapiro's Lemma, we then have
\[
H^1(\bQ_p,T_\infty)\simeq H^1_{\rm Iw}(k_\infty,T),
\]
and the assumption that $T^{G_{k_\infty}}=\{0\}$ implies that
$H^1(\bQ_p,T_\infty)$ is torsion-free. Therefore,
the exact sequence of $\cO_L\pwseries{{\rm Gal}(k_\infty/\bQ_p)}$-modules
\[
0\longrightarrow T_\infty\xrightarrow{\gamma-1}T_\infty
\longrightarrow T\longrightarrow 0
\]
induces the cohomology exact sequence
\[
0\longrightarrow H^1(\bQ_p,T_\infty)\xrightarrow{\gamma-1}H^1(\bQ_p,T_\infty)
\longrightarrow H^1(\bQ_p,T),
\]
giving the proof of the first claim, and the second follows from an
immediate calculation.
\end{proof}

Let $h=h_K$ be the class number of $K$, write $\pp^{h}=\pi_\pp\cO_K$,
and set $\varpi_\pp=\pi_\pp/\overline{\pi}_{\pp}\in K_\pp^\times\simeq\bQ_p^\times$.
Define
\begin{equation}\label{differenceLinv}
\mathscr{L}_\pp(f,K):=\mathscr{L}_p(f)-\mathscr{L}_\pp(\chi_K),
\end{equation}
where $\mathscr{L}_p(f)$ is the $\mathscr{L}$-invariant of $f$ (as defined in \cite[\S{II.14}]{mtt}, for example), and
\[
\mathscr{L}_\pp(\chi_K):=\frac{\log_p(\varpi_\pp)}{{\rm ord}_p(\varpi_\pp)}=-\frac{2\log_p(\overline{\pi}_{\pp})}{h}
\]
is the $\mathscr{L}$-invariant of the quadratic character $\chi_K$ associated to $K$ (see \cite[\S{1}]{greenberg-zeros}, for example),
with ${\rm ord}_p$ the $p$-adic valuation on $\bQ_p$ with the normalization
${\rm ord}_p(p)=1$.
\sk

The following derivative formula is the main result of this paper.

\begin{thm}\label{main}
Let $f\in S_2(\Gamma_0(Np))$ be a newform split multiplicative at $p$,
let $\sF\in\mathbb{I}[[q]]$ be the Hida family passing through $f$,
let $\mathfrak{Z}_\infty\in H_{\rm Iw}^1(K_\infty,\mathbf{T}^\dagger)$
be Howard's system of big Heegner points, and define
$\mathcal{Z}_{\pp,f,\infty}:=\{\mathcal{Z}_{\pp,f,n}\}_{n\geq 0}\in H^1_{\rm Iw}(K_{\infty,\pp},\fil^+V_f)$ by
\[
\mathcal{Z}_{\pp,f,n}:={\rm loc}_\pp(\nu_f(\mathfrak{Z}_n)),
\]
where $\nu_f\in\mathcal{X}_{\cO_L}^a(\mathbb{I})$ is such that $f=\nu_f(\mathbf{f})$. 
Then $\mathcal{Z}_{\pp,f,0}=0$ and 
\begin{equation}\label{eq:deriv}
\mathcal{Z}_{\pp,f,0}'
=\mathscr{L}_{\pp}(f,K)\cdot{\rm loc}_\pp(\kappa_f),
\end{equation}
where $\mathscr{L}_{\pp}(f,K)$ is the $\mathscr{L}$-invariant $(\ref{def:Linv})$, and 
$\kappa_f\in H^1(K,V_f)$ is the image of the degree zero divisor $(A,A[\mathfrak{Np}])-(\infty)$ under the
Kummer map $(\ref{eq:kummer})$.
\end{thm}

\begin{proof}
%
%
By Proposition~\ref{prop:improved}, Corollary~\ref{cor:factor}, Theorem~\ref{thm:equality},
and Theorem~\ref{prop:bigL}, respectively, we see that
\begin{align*}
\left(1-a_p(f)p^{-1}\right)\langle{\rm log}_{}(\mathcal{Z}_{\pp,f,0}),\omega_f\rangle_{}
&=\lim_{\nu\to\nu_f}\nu\left(\bar{\mathcal{L}}^\omega_{\fil^+\bT}({\rm loc}_\pp(\mathfrak{Z}_0^{\xi^{-1}}))\right)\\
&=\lim_{\nu\to\nu_f}\left(1-\frac{\Theta_\nu^{-1}\xi_\nu^{}({\rm Fr}_\qq)}{\nu(\mathbf{a}_p)}\right)
\mathcal{L}_{\fil^+\bT}^\omega({\rm loc}_{\pp}(\mathfrak{Z}_\infty^{\xi^{-1}}))(\nu,\mathds{1})\\
&=\lim_{\nu\to\nu_f}\left(1-\frac{\Theta_\nu^{-1}\xi_\nu^{}({\rm Fr}_\qq)}{\nu(\mathbf{a}_p)}\right)
L_{\pp,\xi}(\mathbf{f})(\nu,\mathds{1})\\
&=\left(1-a_p(f)^{-1}\right) L_\pp(f,{\rm\mathbf{N}}_K).
\end{align*}

Since $a_p(f)=1$ by hypothesis, this shows that
$\langle{\rm log}_{}(\mathcal{Z}_{\pp,f,0}),\omega_f\rangle_{}=0$,
and the vanishing of $\mathcal{Z}_{\pp,f,0}$ follows. Now 
to the proof of the derivative formula $(\ref{eq:deriv})$.
 
Denote by $L_{\pp,\xi}(\sF)^\iota$ the image of $L_{\pp,\xi}(\sF)$
under the involution of $\tilde{\cR}\pwseries{\Gamma_\infty}$ induced by complex conjugation,
so that $L_{\pp,\xi}(\sF)^\iota(\chi)=L_{\pp,\xi}(\sF)(\chi^{-1})$
for every character $\chi$ of $\Gamma_\infty$. One immediately checks the commutativity of the diagram
\[
\xymatrix{
H^1_{\rm Iw}(K_\infty,\fil^+\bT)\ar[rr]^-{{\rm loc}_\pp}\ar[d]^{*}
&& H^1_{\rm Iw}(K_{\infty,\pp},\fil^+\bT) \ar[rr]^-{\mathcal{L}_{\fil^+\mathbb{T}}^{\omega}}\ar[d]^{*}
&& \tilde{\cR}\pwseries{\Gamma_{\pp,\infty}} \ar[d]^{\iota}
\\
H^1_{\rm Iw}(K_\infty,\fil^+\bT)\ar[rr]^-{{\rm loc}_{\overline\pp}}
&& H^1_{\rm Iw}(K_{\infty,\overline{\pp}},\fil^+\bT) \ar[rr]^-{\mathcal{L}_{\fil^+\mathbb{T}}^{\omega}}
&& \tilde{\cR}\pwseries{\Gamma_{\overline{\pp},\infty}},
}
\]
where the left and middle vertical arrows denote the action of complex conjugation.


Define the functions
\begin{equation}\label{defL}
\mathcal{L}_\pp(k,t):=
\left(1-\frac{(p/\varpi_{\pp})^{k/2-1}}{\nu_k(\mathbf{a}_p)\varpi_{\pp}^{-t}}\right)
L_{\pp,\xi}(\mathbf{f})(\nu_k,\phi_o^t),\quad
\mathcal{L}_{\overline{\pp}}(k,t)
:=\left(1-\frac{(p\varpi_{\pp})^{k/2-1}}{\nu_k(\mathbf{a}_p)\varpi_{\pp}^{t}}\right)
L_{\pp,\xi}(\mathbf{f})(\nu_k,\phi_o^{-t}),\nonumber
\end{equation}
where $\phi_o$ is the character $(\ref{phi})$.
By the combination of Theorem~\ref{cor:L-} and Theorem~\ref{thm:equality}, 
we then have
\begin{equation}\label{espL}
\mathcal{L}_\pp(k,t)=
\frac{1}{t!}\left(1-\frac{\nu_k(\mathbf{a}_p)\varpi_{\pp}^{-t}}{p(p/\varpi_{\pp})^{k/2-1}}\right)
\langle\log_{}({\rm loc}_\pp(\nu_k(\mathfrak{Z}_\infty)_{1-k/2+t})),\breve{\omega}_{\nu_k}\rangle_{},\nonumber
\end{equation}
and by the above diagram we also have
\begin{equation}\label{esL*}
\mathcal{L}_{\overline{\pp}}(k,t)=
\frac{1}{t!}\left(1-\frac{\nu_k(\mathbf{a}_p)\varpi_{\pp}^{t}}{p(p\varpi_{\pp})^{k/2-1}}\right)
\langle\log_{}({\rm loc}_{\overline{\pp}}(\nu_k(\mathfrak{Z}^*_\infty)_{k/2-1-t})),\breve{\omega}_{\nu_k}\rangle_{}.\nonumber
\end{equation}


By the ``functional equation'' satisfied by $\mathfrak{Z}_\infty$ (see Theorem~\ref{thm:bigHPs}),
it follows that the function
\[
\mathcal{L}_p(k,t):=\mathcal{L}_\pp(k,t)-w\mathcal{L}_{\overline{\pp}}(k,k-2-t)
\]
vanishes identically along the ``line'' $t=k/2-1$. By \cite[Prop.~2.3.6]{howard-invmath}, the sign 
$w$ is the \emph{opposite} of the sign in the functional equation for the $p$-adic $L$-function $L_p^{}(f,s)$ 
associated to $f$ in \cite{mtt}. Thus, if $w=1$, then ${\rm ord}_{s=1}L_p(f,s)>2$, 
and by \cite[Lemma~6.1]{venerucci-exp} the right-hand side of $(\ref{eq:deriv})$ vanishes; since the 
vanishing of the left-hand side follows easily from the construction of $\mathcal{Z}'_{\gamma,\infty}$ 
in Lemma~\ref{lem:divide}, we conclude that $(\ref{eq:deriv})$ reduces to the identify ``$0=0$'' when $w=1$. 
As a consequence, in the following we shall assume that $w=-1$.

Using the formula for the $\mathscr{L}$-invariant of $f$ as the logarithmic derivative
of $\nu_k(\mathbf{a}_p)$ at $k=2$ (see \cite[Thm.~3.18]{GS}, for example) 
and noting that $(p/\varpi_\pp)^{k/2-1}=\overline{\pi}_\pp^{(k-2)/h}$ by definition, 
we find
\begin{align}\label{eq:LHS}
\frac{\partial}{\partial k}\mathcal{L}_p(k,t)\bigr\vert_{(2,0)}
&=\left[\frac{d}{dk}\nu_k(\mathbf{a}_p)\bigr\vert_{k=2}-\frac{\log_p(\overline{\pi}_{\pp})}{h}
-w\left(\frac{d}{dk}\nu_k(\mathbf{a}_p)\bigr\vert_{k=2}-\frac{\log_p(\overline{\pi}_{\pp})}{h}\right)\right]
L_\pp(f)({\rm\mathbf{N}}_K)\\
&=-\left[\frac{(1-w)}{2}\left(\mathscr{L}_p(f)-\mathscr{L}_\pp(\chi_K)\right)\right]L_\pp(f)({\rm\mathbf{N}}_K)\nonumber\\
&=-\mathscr{L}_\pp(f,K)\cdot L_\pp(f)({\rm\mathbf{N}}_K).\nonumber
\end{align}
Using the aforementioned vanishing of $\mathcal{L}_p(k,k/2-1)$ for the first equality,  
we also find 
\begin{align}\label{eq:RHS}
\frac{\partial}{\partial k}\mathcal{L}_p(k,t)\bigr\vert_{(2,0)}
=-\frac{1}{2}\frac{\partial}{\partial t}\mathcal{L}_p(k,t)\bigr\vert_{(2,0)}
&=-\frac{(1-w)}{2}\left(1-a_p(f)p^{-1}\right)\langle{\rm log}_{}(\mathcal{Z}_{\pp,f,0}'),\omega_{f}\rangle_{}\\
&=-(1-p^{-1})\langle{\rm log}_{}(\mathcal{Z}_{\pp,f,0}'),\omega_f\rangle_{},\nonumber
\end{align}
and comparing $(\ref{eq:LHS})$ and $(\ref{eq:RHS})$, we arrive at the equality
\begin{equation}\label{eq:impr}
(1-p^{-1})\langle{\rm log}_{}(\mathcal{Z}_{\pp,f,0}'),\omega_{f^{}}\rangle_{}
=\mathscr{L}_\pp(f,K)\cdot L_\pp(f)({\rm\mathbf{N}}_K).
\end{equation}

On the other hand, letting $\varphi_0:A\longrightarrow A$ be the identity isogeny,
by Theorem~\ref{thmbdp1A} we have
\begin{align*}\label{eq:HP}
L_\pp(f)(\mathbf{N}_K)
&=(1-a_p(f)p^{-1})\sum_{[\mathfrak{a}]\in{\rm Pic}(\cO_K)}\langle{\rm AJ}_{F}(\Delta_{\varphi_\fa\varphi_0}),\omega_{f}\rangle_{}\\
&=(1-p^{-1})\langle{\log}_{}({\rm loc}_\pp(\kappa_f)),\omega_{f}\rangle_{},
\end{align*}
which combined with $(\ref{eq:impr})$ concludes the proof of Theorem~\ref{main}.
\end{proof}

\begin{rem}
It would be interesting to extend the main result of this paper to higher weights.
As is well-known (see \cite[Thm.~3]{Li}, for example), if $f
\in S_{k}(\Gamma_0(Np))$ is a newform with $U_p$-eigenvalue $a_p(f)$, 
then $a_p(f)^2=p^{k-2}$. 
Thus if $k>2$, then $f$ has positive slope (i.e., it is \emph{not} $p$-ordinary), 
and the extension of our Theorem~\ref{main} to this case 
would require an extension to Coleman families\footnote{For a recent result along these lines
(albeit for a different Euler system), see \cite{LZ-Coleman}.}
of Howard's construction of big Heegner
points in Hida families \cite{howard-invmath}.
\end{rem}

\bibliographystyle{alpha}
\bibliography{Heegner}

\end{document}